\documentclass[a4paper,11pt]{article}
\usepackage[utf8]{inputenc}

\usepackage{amssymb,amsmath,amsthm}
\usepackage{epsfig}
\usepackage{breqn}
\usepackage{rotating}
\usepackage{amsfonts}
\usepackage{setspace}
\usepackage{fullpage}
\usepackage{enumitem}
\usepackage{bbold} 
\usepackage{amsmath}
\usepackage{comment}
\usepackage{pgf,tikz}
\usepackage{mathtools}  
\usepackage{graphicx}
\numberwithin{equation}{section}

\usepackage{setspace}
\newtheorem{theorem}{Theorem}
\newtheorem{claim}[theorem]{Claim}
\newtheorem{lemma}[theorem]{Lemma}

\newtheorem{definition}[theorem]{Definition}

\newtheorem{observation}{Observation}
\newtheorem{summary}{Summary}

\newcommand{\floor}[1]{\left\lfloor{#1}\right\rfloor}

\def\C{\mathcal C}
\def\B{\mathcal B}
\def\h{\mathcal H}

\def\S{\mathcal S}

\def\F{\mathcal F}
\def\E{\mathcal E}
\def\G{\mathcal G}

\title{Stability of extremal connected hypergraphs avoiding Berge-paths}

\author{D\'aniel Gerbner$^1$ \and D\'aniel T. Nagy$^1$ \and Bal\'azs Patk\'os$^{1,2}$ \and Nika Salia$^{3}$ \and M\'at\'e Vizer$^1$
\\
\small $^1$ Alfr\'ed R\'enyi Institute of Mathematics, Budapest \\
\small $^2$ Moscow Institute of Physics and Technology\\
\small $^3$ Department of Mathematics, King Fahd University of Petroleum and Minerals.}

\date{}

\begin{document}
 
\maketitle

\begin{abstract}
A Berge-path of length $k$ in a hypergraph $\h$ is a sequence $v_1,e_1,v_2,e_2,\dots,v_{k},e_k,v_{k+1}$ of distinct vertices and hyperedges with  $v_{i},v_{i+1} \in e_i$, for $i \le k$. F\"uredi, Kostochka and Luo, and independently Gy\H{o}ri, Salia and Zamora determined the maximum number of hyperedges in an $n$-vertex, connected,  $r$-uniform hypergraph that does not contain a Berge-path of length $k$  provided $k$ is large enough compared to $r$. They also determined the unique extremal hypergraph $\h_1$.

We prove a stability version of this result by presenting another construction $\h_2$ and showing that any $n$-vertex, connected, $r$-uniform hypergraph without a Berge-path of length $k$, that contains more than $|\h_2|$ hyperedges must be a sub-hypergraph of the extremal hypergraph $\h_1$, provided $k$ is large enough compared to $r$.
\end{abstract}

\section{Introduction}
In extremal graph theory, the Tur\'an number $ex(n,G)$ of a graph $G$ is the maximum number of edges that an $n$-vertex graph can have without containing $G$ as a subgraph. If a class $\G$ of graphs is forbidden, then the Tur\'an number is denoted by $ex(n,\G)$. The asymptotic behavior of the function $ex(n,G)$ is well-understood if $G$ is not bipartite. However, much less is known if $G$ is bipartite (see the survey \cite{FurSim}).
One of the simplest classes of bipartite graphs is that of paths. Let $P_k$ and $C_k$ denote the path and the cycle with $k$ edges and let $\C_{\ge k}$ denote the class of cycles of length at least $k$.

Erd\H{o}s and Gallai \cite{Er-Ga} proved that for any  $n \geq k \geq 1$, the Tur\'an number satisfies 
$ex(n,P_k)\leq \frac{(k - 1)n}{2}$. They obtained this result by first showing
that for any $n \geq k \geq 3$, $ex(n,\C_{\geq k}) \leq \frac{(k - 1)(n - 1)}{2}$. The bounds are sharp for paths if $k$ divides $n$, and sharp for cycles, if $k-2$ divides $n-1$. These are shown by the example of $n/k$ pairwise disjoint $k$-cliques for the path $P_k$, and adding an extra vertex joined by an edge to every other vertex for the class $\C_{\ge k+2}$ of cycles. Later, Faudree and Schelp \cite{FSch} gave the exact value of $ex(n,P_k)$ for every $n$.

Observe that the extremal construction for the path is not connected.   Kopylov \cite{Kopylov} and independently Balister, Gy\H ori, Lehel, and Schelp \cite{BGLS} determined the maximum number of edges $ex^{conn}(n,P_k)$ that an $n$-vertex connected graph can have without containing a path of length $k$. The stability version of these results was proved by F\"uredi, Kostochka and Verstra\"ete \cite{furedi2016stability}. To state their result, we need to define the following class of graphs. 

\begin{definition}\label{defforgraphs} For $n \ge k$ and $\frac{k}{2} > a \ge 1$ we define the graph $H_{n,k,a}$ as follows. The vertex set of $H_{n,k,a}$ is partitioned into three disjoint parts $A,B$ and $L$ such that $|A|=a$, $|B|=k-2a$ and $|L|=n-k+a$. The edge set of $H_{n,k,a}$ consists of all the edges between $L$ and $A$ and also all the edges in $A \cup B$. Let us denote the number of edges in $H_{n,k,a}$ by $|H_{n,k,a}|$.
 
\end{definition}

\begin{theorem}[F\"uredi, Kostochka, Verstra\"ete \cite{furedi2016stability}, Theorem 1.6]\label{stabgraph} Let $t\ge 2$, $n \ge 3t-1$ and $k \in \{2t, 2t+1 \}$. Suppose we have a $n$-vertex connected $P_{k-1}$-free graph $G$ with more edges than $|H_{n+1,k+1,t-1}|-n$. Then we have either

\smallskip 

$\bullet$ $k=2t$, $k\neq 6$ and $G$ is a subgraph of $H_{n,k,t-1}$, or

\smallskip 

$\bullet$  $k=2t+1$ or $k=6$, and $G \setminus A$ is a star forest for $A \subseteq V(G)$ of size at most $t-1$.

\end{theorem}

The Tur\'an numbers for hypergraphs $ex_r(n,\h)$, $ex_r(n,\mathbb{H})$ can be defined analogously for $r$-uniform hypergraphs $\h$ and classes $\mathbb{H}$ of $r$-uniform hypergraphs. Note that there are several ways how one can define paths and cycles of higher uniformity. In this paper, we consider the definition due to Berge.

\begin{definition}
A \emph{Berge-path} of length $t$  is an alternating sequence of $t+1$ distinct vertices and $t$ distinct hyperedges of the hypergraph,  $v_1, e_1, v_2, e_2, v_3, \dots, e_t, v_{t+1}$ such that $v_{i},v_{i+1} \in e_i$, for $i \in [t]$. The vertices $v_1, v_2, \dots, v_{t+1}$ are called \emph{defining vertices} and the hyperedges $e_1,e_2,\dots,e_t$ are called \emph{defining hyperedges} of the Berge-path. We denote the set of all Berge-paths of length $t$ by $\B P_t$.  

Similarly, a \emph{Berge-cycle} of length $t$ is an alternating sequence of $t$ distinct vertices and $t$ distinct hyperedges of the hypergraph,  $v_1, e_1, v_2, e_2, v_3, \dots, v_t, e_t$, such that $v_{i},v_{i+1} \in e_i$, for $i\in [t]$, where indices are taken modulo $t$. The vertices $v_1, v_2, \dots, v_{t}$ are called \emph{defining vertices} and the hyperedges $e_1,e_2,\dots,e_t$ are called \emph{defining hyperedges} of the Berge-cycle.

As these are the only cycles and paths we consider in hypergraphs, we will often omit the word Berge.
\end{definition}
 
The study of the Tur\'an numbers $ex_r(n,\B P_k)$ was initiated  by Gy\H ori, Katona and Lemons~\cite{GyoKaLe}, who determined the quantity in almost every case. Later Davoodi, Gy\H ori, Methuku and Tompkins \cite{DavoodiGMT} settled the missing case $r=k+1$. For results on 
the maximum number of hyperedges in $r$-uniform hypergraphs not containing Berge-cycles longer than $k$ see \cite{furedi2018avoiding,GLSZ} and the references therein. 

Analogously to graphs,  a hypergraph is \emph{connected}, if for any two of its vertices, there is a Berge-path containing both vertices. The connected Tur\'an numbers for an   $r$-uniform hypergraph $\h$ and class of  $r$-uniform hypergraphs $\mathbb{H}$ can be defined analogously, they are denoted by the functions $ex_r^{conn}(n,\h)$ and
$ex_r^{conn}(n,\mathbb{H})$, respectively.

To describe the extremal result and to introduce our contributions, we need the following definition that can be considered as an analogue of Definition~\ref{defforgraphs} for higher uniformity.

\begin{definition}
 For integers $n,a \ge 1$ and $b_1, \dots,b_t \ge 2$ with $n \ge 2a + \sum_{i=1}^{t}b_i$ and $t\geq 0$ let us denote by $\h_{n,a,b_1,b_2,\dots,b_t}$ the following $r$-uniform hypergraph, see Figure~\ref{Fig:H_n_a}. 
 
 \smallskip 
 
 $\bullet$  Let the vertex set of $\h_{n,a,b_1,b_2,\dots,b_t}$ be $A\cup L \cup \bigcup_{i=1}^t B_i$, where $A,B_1,B_2,\dots,B_t$ and $L$ are pairwise disjoint sets of sizes $|A|=a$, $|B_i|=b_i$ ($i=1,2,\dots,t$) and $|L|= n-a-\sum_{i=1}^tb_i$.
 
 \smallskip 
 
 $\bullet$ Let the hyperedges of $\h_{n,a,b_1,b_2,\dots,b_t}$ be
 $$
 \binom{A}{r}\cup \bigcup_{i=1}^t\binom{A \cup B_i}{r}\cup \left\{\{c\}\cup A': c\in L,A'\in \binom{A}{r-1}\right\}.
 $$
 
\end{definition}

Observe that the number of hyperedges in $\h_{n,a,b_1,b_2,\dots,b_t}$ is $$\left(n-a-\sum_{i=1}^tb_i\right)\binom{a}{r-1}+\sum_{i=1}^t\binom{a+b_i}{r}-(t-1)\binom{a}{r},$$
and in case $t=0$ the vertex set of the graph $\h_{n,a}$ consists of two sets $A$ and $L$ and the number of hyperedges is 
$$\left(n-a-\sum_{i=1}^tb_i\right)\binom{a}{r-1}+\binom{a}{r}.$$

\begin{figure}[ht]
\begin{center}
\includegraphics[width=0.5\textwidth]{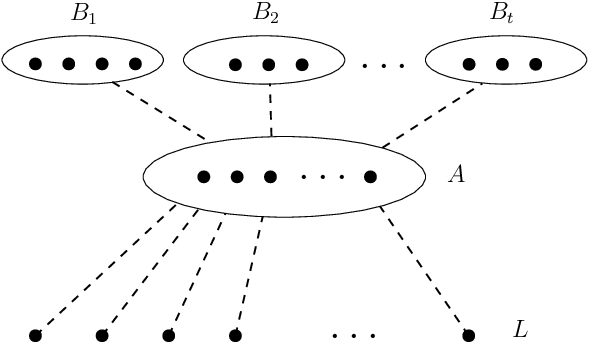}
\caption{The hypergraph $\h_{n,a,b_1,b_2,\dots,b_t}$.}
\end{center}
\label{Fig:H_n_a}
\end{figure}

Note that, if $a\le a'$ and $b_i\le b_i'$ for all $i=1,2,\dots,t$, then $\h_{n,a,b_1,b_2,\dots,b_t}$ is a sub-hypergraph of $\h_{n,a',b_1',b_2',\dots,b_t'}$. Finally, the length of the longest path in $\h_{n,a,b_1,b_2,\dots,b_t}$ is $2a-t+\sum_{i=1}^tb_i$ if $t\le a+1$, and $a-1+\sum_{i=1}^{a+1}b_i$ if $t>a+1$ and the $b_i$'s are in non-increasing order.

With a slight abuse of notation, we define $\h_{n,a}^+$ to be a hypergraph obtained from $\h_{n,a}$ by adding an arbitrary hyperedge.
Hyperedges containing at least $r-1$ vertices from $A$ are already in $\h_{n,a}$, therefore there are $r-1$ pairwise different hypergraphs that we denote by $\h_{n,a}^+$ depending on the number of vertices from $A$ in the extra hyperedge. Observe that the length of the longest path in $\h_{n,a}^+$ is one larger than in $\h_{n,a}$, in particular, if $k$ is even, then $\h_{n,\floor{\frac{k-1}{2}}}^+$ does not contain a Berge-path of length $k$.

The first attempt to determine the largest number of hyperedges in connected $r$-uniform hypergraphs without a Berge-path of length $k$ can be found in \cite{GyoriMSTV}, where  the asymptotics of the extremal function was determined. 
The Tur\'an number of Berge-paths in connected hypergraphs was determined by F\"uredi, Kostochka and Luo \cite{furedi2019connpath} for $ k \ge 4r \ge 12$ and $n$ large enough. Independently in a different range, it was also given by Gy\H{o}ri, Salia and Zamora \cite{GSZ}, who also proved the uniqueness of the extremal structure. To state their result, let us introduce the following notation: for a hypergraph $\h$ we denote by $|\h|$ the numbers of hyperedges in $\h$.

\begin{theorem}[Gy\H{o}ri, Salia, Zamora,  
\cite{GSZ}]\label{GSZ}
For all integers  $k,r$ with $k \geq 2r+13 \geq 18$ there exists $n_{k,r}$ such that if $n>n_{k,r}$, then we have
\begin{itemize}
    \item 
$ex_r^{conn}(n,\B P_k)=|\h_{n,\floor{\frac{k-1}{2}}}|$, if $k$ is odd, and 
\item $ex_r^{conn}(n,\B P_k)=|\h_{n,\floor{\frac{k-1}{2}},2}|$, if $k$ is even. 
\end{itemize}
Depending on the parity of $k$, the unique extremal hypergraph is $\h_{n,\floor{\frac{k-1}{2}}}$ or $\h_{n,\floor{\frac{k-1}{2}},2}$.
\end{theorem}

Our main result provides a stability version (and thus a strengthening) of Theorem~\ref{GSZ} and also an extension of Theorem~\ref{stabgraph} for  uniformity at least 3.

First, we state it for hypergraphs with the minimum degree  at least 2, and then in full generality. In the proof, the hypergraphs $\h_{n,\frac{k-3}{2},3}$ and $ \h_{n,\frac{k-3}{2},2,2}$ will play a crucial role in case $k$ is odd, while if $k$ is even, then the hypergraphs $\h_{n,\lfloor\frac{k-3}{2}\rfloor,4}, \ \h_{n,\lfloor\frac{k-3}{2}\rfloor,3,2}$ and $\h_{n,\lfloor\frac{k-3}{2}\rfloor,2,2,2}$ will be of importance, note that all of them are $n$-vertex, maximal, $\B P_k$-free hypergraphs. In both cases, the hypergraph listed first contains the largest number of hyperedges. This number gives the lower bound in the following theorem.

\begin{theorem}\label{deg2}
For any $\varepsilon >0$ there exist integers $q=q_{\varepsilon}$ and $n_{k,r}$ such that if $r\ge 3$, $k\ge(2+\varepsilon)r+q$, $n\ge n_{k,r}$ and $\h$ is a connected  $n$-vertex, $r$-uniform hypergraph with minimum degree at least 2, without a Berge-path of length $k$, then we have the following.
\begin{itemize}
    \item 
    If $k$ is odd and $|\h|>|\h_{n,\frac{k-3}{2},3}|=(n-\frac{k+3}{2})\binom{\frac{k-3}{2}}{r-1}+\binom{\frac{k+3}{2}}{r}$, then $\h$ is a sub-hypergraph of $\h_{n,\frac{k-1}{2}}$.
    \item
    If $k$ is even and 
$|\h|> |\h_{n,\lfloor\frac{k-3}{2}\rfloor,4}|=(n-\lfloor\frac{k+5}{2}\rfloor)\binom{\lfloor\frac{k-3}{2}\rfloor}{r-1}+\binom{\lfloor\frac{k+5}{2}\rfloor}{r}$, then $\h$ is a sub-hypergraph of $\h_{n,\lfloor\frac{k-1}{2}\rfloor,2}$ or $\h_{n,\floor{\frac{k-1}{2}}}^+$.
\end{itemize}
\end{theorem}
Based on our analysis, we posit that there exists a small constant $q$ such that Theorem~\ref{deg2} holds for all $k\geq 2r+q$, indicating that the requirement $k\geq (2+\varepsilon)r+q$ is not the optimal condition. However, it is worth noting that the constant $2$ preceding $r$ is necessary for the existence of constructions.

Let $\mathbb{H}'_{n',a,b_1,b_2,\dots,b_t}$ be the class of hypergraphs that can be obtained from $\h_{n,a,b_1,b_2,\dots,b_t}$ for some $n\le n'$ by adding hyperedges of the form $A'_j\cup D_j$, where the $D_j$'s partition $[n']\setminus [n]$, all $D_j$'s are of size at least 2 and $A_j' \subseteq A$ for all $j$. Let us define $\mathbb{H}_{n',\floor{\frac{k-1}{2}}}^+$ analogously.

\begin{theorem}\label{general}
For any $\varepsilon >0$ there exist  integers $q=q_{\varepsilon}$ and $n_{k,r}$ such that if $r\ge 3$, $k\ge (2+\varepsilon)r+q$, $n\ge n_{k,r}$ and
$\h$ is a connected  $n$-vertex, $r$-uniform hypergraph without a Berge-path of length $k$, then we have the following.
\begin{itemize}
    \item 
    If $k$ is odd and $|\h|>|\h_{n,\frac{k-3}{2},3}|$, then $\h$ is a sub-hypergraph of some $\h'\in \mathbb{H}'_{n,\frac{k-1}{2}}$.
    \item
    If $k$ is even and 
$|\h|> |\h_{n,\lfloor\frac{k-3}{2}\rfloor,4}|$, then $\h$ is a sub-hypergraph of some $\h'\in \mathbb{H}'_{n,\lfloor\frac{k-1}{2}\rfloor,2}$ or $\mathbb{H}_{n,\floor{\frac{k-1}{2}}}^+$.
\end{itemize}
\end{theorem}

\medskip 

\subsection*{Notation} We use standard notation. The vertex set of an $r$-uniform hypergraph $\h$ is denoted by $V(\h)$, and we denote the set of its hyperedges by $\E(\h)$. Sets of vertices that do not necessarily form a hyperedge are denoted by capital letters, while lowercase letters $u,v,x,y,z$ are used to denote vertices. Hyperedges of $\h$  are usually denoted by lowercase $h$. Hyperedges of some particular cycles are often denoted by $e$ and those of paths by $f$. For a hypergraph $\h$ and a set $S$ of its vertices, the set of hyperedges of $\h$ that contain at least one element of $S$ is denoted by $E(S)$ ($\h$ will always be clear from context). The (open) neighborhood of a vertex $v$ in $\h$ (i.e., the set of vertices $u$ different from $v$ for which there exists a hyperedge $h$ of $\h$ with $\{u,v\}\subset h$) is denoted by $N_\h(v)$ or simply $N(v)$ if $\h$ is clear from context. For two hypergraphs $\h_1$ and $\h_2$ with $V(\h_2)\subseteq V(\h_1)$, we denote by $\h_1\setminus \h_2$ the hypergraph with vertex set $V(\h_1)$ and hyperedge set $\E(\h_1)\setminus \E(\h_2)$. For two hypergraphs $\h_1$ and $\h_2$ we denote the fact that $\h_1$ is sub-hypergraph of $\h_2$ by $\h_1 \subseteq \h_2$.

\subsection*{An informal outline of the proof}

We first prove Theorem~\ref{deg2} with an additional degree constraint, a lower bound on the number of hyperedges intersecting any small set of vertices. The majority of the paper deals with this additional assumption; in the last page we show how this easily implies both of our theorems. Most of the proof is through technical claims that establish some properties of certain vertices and hyperedges, but are not interesting on their own.

Let us describe the basic idea of the proof. Observe that in $\h_{n,a,b_1,b_2,\dots,b_t}$, in the longest paths and cycles the vertices of $B_i$ form intervals, while the rest of the path or cycle alternates between $A$ and $L$. Our approach is to take a large connected $n$-vertex hypergraph $\h$ without a Berge path of length $k$, and show that it is similar to $\h_{n,a,b_1,b_2,\dots,b_t}$. In order to do that, we first show that $\h$ contains a Berge cycle of length at least $k-4$ (Claim~\ref{Contains_cycle}). Then we show that some of the vertices of this cycle can be replaced by some of the outside vertices. 
These vertices and the outside vertices form the set $L$, the intervals among the remaining vertices of the cycle form the sets $B_i$ and the rest of the vertices form $A$.

Afterwards, we have to deal with the hyperedges that are not in the copy of $\h_{n,a,b_1,b_2,\dots,b_t}$ obtained this way. We show that either there are no such hyperedges, or there are no too many such hyperedges and $\h_{n,a,b_1,b_2,\dots,b_t}$ is small. 

A sufficient condition for the property of being replaceable is shown in Claim~\ref{replace}, and afterwards it is shown that each outside vertex can replace several vertices of the cycle. We fix an outside vertex $w$, and Claim 14 shows some properties of the intervals of vertices not replaceable by $w$.

After this, we summarize what the technical claims imply about the hyperedges that are not in the above defined $\h_{n,a,b_1,b_2,\dots,b_t}$. The proof is finished by a case by case analysis, where the cases depend on the length of the longest cycle and the structure of the intervals of non-replaceable vertices.

\section{Proof of Theorem~\ref{deg2} with an additional degree constraint}
We say that an $r$-uniform  hypergraph $\h$ has the \textbf{set degree condition}, if, for any set $X$ of vertices with $|X|\le k/2$, we have $|E(X)| \ge|X|\binom{\lfloor \frac{k-3}{2} \rfloor }{r-1}$, i.e. the number of those hyperedges that are incident to some vertex in $X$ is at least $|X|\binom{\lfloor \frac{k-3}{2} \rfloor }{r-1}$. We first prove Theorem~\ref{deg2} for such hypergraphs.
To this end, we prove the following technical lemma that will be crucial at the later stage of the proof.

\begin{lemma}\label{structure}
Let $\h$ be a connected $r$-uniform hypergraph with the minimum degree at least 2 and with longest Berge-path and Berge-cycle of length $\ell-1$. Let $C$ be a Berge-cycle of length $\ell-1$ in $\h$, with defining vertices $V=\{v_1,v_2,\dots,v_{\ell-1}\}$ and defining edges $\E(C)=\{e_1,e_2,\dots,e_{\ell-1}\}$ with $v_i,v_{i+1}\in e_i$ (modulo ${\ell-1}$). Then, we have

\medskip 

(i) every hyperedge $h \in \h\setminus C$ contains at most one vertex from $V(\h)\setminus V$.

\smallskip 

(ii) If $u,v$ are not necessarily distinct vertices from $V(\h)\setminus V$, then there cannot exist distinct hyperedges $h_1,h_2\in \h\setminus \C$ and an index $i$ with $v,v_i \in h_1$ and $u,v_{i+1}\in h_2$.

\smallskip 

(iii) If there exists a vertex $v\in V(\h)\setminus V$ and there exist different hyperedges $h_1,h_2\in \h\setminus \C$ with $v,v_{i-1} \in h_1$ and $v,v_{i+1}\in h_2$, then there exists a cycle of length ${\ell-1}$ not containing $v_i$ as a defining vertex.
\end{lemma}

\begin{proof}
We prove (i) by contradiction. Suppose $h \in \h\setminus C$ contains 
two vertices from $V(\h)\setminus V$. We distinguish two cases.

\medskip 
\textbf{Case 1.}
Hyperedge $h$ contains a vertex $u\not\in V$ and a different vertex $v\in e_i\setminus V$ for some $i\le \ell-1$. Then $v_{i+1},e_{i+1},v_{i+2},\dots,v_{\ell},e_\ell,v_1,e_1,\dots,v_i,e_i,v,h,u$ is a path of length $\ell$, a contradiction. 

\smallskip 

\textbf{Case 2.}
Hyperedge $h$ contains two vertices $u$ and $v$ from $V(\h)\setminus V(C)$. We consider the hypergraph $\h'$  obtained from $\h$ by removing the hyperedge $h$.

\smallskip 

\hspace{2mm} \textbf{Case 2.1.} There is a Berge-path in $\h'$ from $\{v,u\}$ to the cycle $C$, in particular to a defining vertex of $C$. Then let $P$ be a shortest such path, let us assume $P$ is from $v$ to $v_i$. Without loss of generality we may suppose that $P$ does not contain $e_i$ as a defining hyperedge, (it is possible $P$ contains $e_{i-1}$ as a defining hyperedge). Then $u,h,P,e_i,v_{i+1},\dots,e_{i-2},v_{i-1}$ is a Berge-path of length at least $\ell$, contradicting the assumption that the longest path in $\h$ is of length $\ell-1$.

\smallskip 

\hspace{2mm} \textbf{Case 2.2.} Suppose there is no Berge-path from the vertex $v$ to the cycle $C$ in $\h'$. However by the connectivity of $\h$, there is a shortest path $P$ from $v$ to a defining vertex of $C$, say $v_i$ and it does not use any defining hyperedge of $C$ but possibly $e_{i-1}$.   
Also, $h$ is  a hyperedge of $P$. There exists a hyperedge $h'\neq h$ containing $v$, as the minimum degree is at least 2 in $\h$. Note that $h'$ is not a hyperedge of the path $P$, and even more, all vertices of $h'$ different from $v$ are not defining vertices of $P$ or $C$.  Fix a vertex $u'\in h'\setminus \{v\}$. 
Then $u',h',P,e_i,v_{i+1},\dots, e_{i-2},v_{i-1}$ is a Berge-path of length at least $\ell$, a contradiction.

\medskip 

To prove (ii), assume first that $u=v$. Then one could enlarge $C$ by removing $e_i$ and adding $h_1,v,h_2$ to obtain a longer cycle, a contradiction. Assume now $u\neq v$. Then removing $e_i$ and adding $h_1,v$ and $h_2,u$, one would obtain a path of length $\ell$, a contradiction. 

\medskip

Finally to show (iii), we can replace $e_{i-1},v_i,e_i$ in $C$ by $h_1,v,h_2$ to obtain the desired cycle.
\end{proof}

\begin{proof}[Proof of Theorem~\ref{deg2} for hypergraphs having the set degree condition]
Let $\h$ be an $n$-vertex $\B P_k$-free hypergraph with the set degree condition. Also, assume $|\h|$ is as claimed in the statement of the theorem. However, for the most part of proof, we will only use the set degree condition.

\begin{claim}\label{size_of_shadow}
Let $P$ be a longest Berge-path in $\h$ with  defining vertices $U=\{u_1, \dots, u_{\ell}\}$ and defining hyperedges $\F=\{f_1,f_2, \dots, f_{\ell-1}\}$ in this given order.  Suppose $P$ minimizes $x_1+x_\ell$ among longest Berge-paths of $\h$, where $x_i$ for $i\in [\ell]$, denotes the number of hyperedges in $\F$ incident to $u_i$. Then  the sizes of $N_{\h\setminus\F}(u_1)$ and $N_{\h\setminus\F}(u_\ell)$ are at least $ \floor{\frac{k-3}{2}}$. 
\end{claim}

\begin{proof}[Proof of Claim~\ref{size_of_shadow}] Observe that the statement is trivially true for $r\ge 4$ and for the an arbitrary longest path, as by the set degree condition, there exist at least $\binom{\floor{\frac{k-3}{2}}}{r-1}-k+1$ hyperedges in $\h\setminus \F$ incident to $u_1$. This is strictly greater than $\binom{\floor{\frac{k-5}{2}}}{r-1}$ if $r\ge 4$ and $k\ge (2+\varepsilon)r+q$, for large enough $q$, hence $|N_{\h\setminus\F}(u_1)|>\frac{k-5}{2}$.
By the symmetry of the Berge-path we conclude the desired for the vertex $u_{\ell}$ with the same argument, finishing the proof for $r\ge 4$. 

Thus we can assume that $r=3$. Let $P$ be a longest Berge-path in $\h$, minimizing $x_1+x_{\ell}$.
First, we claim that if $u_1 \in f_i$ then $x_i \ge x_1$. Note that the Berge-path 
\begin{displaymath}
u_i, f_{i-1},u_{i-1},f_{i-2}, u_{i-2}, \dots, u_2, f_1, u_1, f_i,u_{i+1}, f_{i+1}, \dots, u_\ell, f_{\ell-1}, u_{\ell}
\end{displaymath}
is also a longest Berge-path, with the same set of defining vertices and defining hyperedges and endpoint $x_{\ell}$, hence by the minimality of the sum $x_1+x_{\ell}$, the number of  hyperedges from $\F$ incident to $u_i$ is at least $x_1$.

This means that if we consider all possible Berge-paths obtained from $P$ by the way described above (including itself), then the number of pairs $(u,f)$, where $u \in U$, $f \in \F$ and $u \in f$, is at least $x_1^2$. On the other hand, this number is upper bounded by $r|\F|=3|\F|= 3(\ell-1) $, hence we have  $x_1^2 \leq 3(\ell-1)  \leq 3(k-1)$, therefore $x_1\leq \sqrt{3(k-1)}$. The same holds for the other end vertex $u_\ell$ and so for $x_{\ell}$ by symmetry.

Since the degree of $u_1$ is at least $\binom{\floor{\frac{k-3}{2}}}{2}$, out of which at most $\sqrt{3(k-1)}$ of the hyperedges are defining hyperedges, the degree of $u_1$ in $\h \setminus \F$ is at least $$ \binom{\floor{\frac{k-3}{2}}}{2} - \sqrt{3(k-1)}
> {\binom{\floor{\frac{k-3}{2}}-1}{2}},$$ 
 if $k\ge 21$. Thus $|N_{\h \setminus \F}(u_1)|\geq \floor{\frac{k-3}{2}}$ and in the same way we have $|N_{\h \setminus \F}(u_{l})|\geq \floor{\frac{k-3}{2}}$. 
\end{proof}

\begin{claim}\label{Contains_cycle}
Let $\ell-1$ be the length of the longest Berge-path in $\h$. Then $\ell \ge k-3$ and $\h$ contains a Berge-cycle of length $\ell-1$.
\end{claim}
\begin{proof}[Proof of Claim~\ref{Contains_cycle}]
Let $u_1,f_1,u_2,f_2,\dots,u_{\ell-1},f_{\ell-1},u_\ell$ be a longest Berge-path given by Claim~\ref{size_of_shadow} with defining hyperedges $\F=\{f_1,f_2,\dots,f_{\ell-1}\}$ and defining vertices $U=\{u_1,u_2,\dots,u_\ell\}$. 

Before the proof let us introduce some notations: for $\E\subseteq \E(\h)$ and integer $j$ with $1 \le j \le \ell$, let $S_{j,\E}$ denote the set of indices of vertices in $U\cap N_{\h\setminus\E}(u_j)$, and we simply denote $S_{j,\F}$ by $S_j$. In particular, $S_j$ denotes the set of indices $i$ such that there is a hyperedge of $\h$ that contains both $u_i$ and $u_j$ and is not a defining hyperedge of the path. For any set $S$ of integers let $S^-:=\{a:a>0,\, a+1\in S\},~S^{--}=(S^-)^-$. The operations $^+$ and $^{++}$  are defined analogously.

To start the proof, observe first that $\h$ cannot contain a Berge-cycle $C$ of length $\ell$. Indeed,  the hyperedges of such a cycle contain at most $\ell(r-1)$ vertices. Therefore 
 there is a vertex $v\in~V(\h)\setminus~V(C)$, then as $\h$ is connected, there exists a path from $v$ to $C$ and we obtain a path of length at least $\ell$, contradicting our assumption on the length of the longest path.

If $\ell\in S_1$ or equivalently $1\in S_\ell$, then a hyperedge showing this, together with $\F$ forms a Berge-cycle of length $\ell$ in $\h$. So we can assume $S_1,S_\ell\subseteq \{2,\dots,\ell-1\}$ and so $S_1^-~\subseteq~\{1,2,\dots,\ell-~1\}$. 

If $ S_1 \cap S^+_\ell \neq \emptyset$ (or symmetrically $S_1^-\cap S_\ell\neq \emptyset$), then $\h$ contains a Berge-cycle of length~$\ell$. Indeed, if $i\in S_1 \cap S^+_\ell$, then there are hyperedges $e$ and $e'$ in $\h\setminus \F$ with $u_1,u_{i+1}\in e$ and  $u_\ell,u_{i}\in e'$. Then 
$$u_{i},f_{i-1},u_{i-1},\dots,f_1,u_1,e,u_{i+1},f_{i+2},\dots,f_{\ell-1},u_\ell,e'$$
is a Berge-cycle of length $\ell$. (Note that $e$ and $e'$ are distinct hyperedges as  $\ell \not \in S_1$.)
Note that by Claim~\ref{size_of_shadow}, we have $|S_\ell|,|S_1^-|\ge \floor{\frac{k-3}{2}}$. Thus since $S_1 \cap S^+_\ell \neq \emptyset$, we have $\ell\ge k-3$.

The exact same argument shows that if $S_1^{--}\cap S_\ell\neq \emptyset$ or symmetrically $S_1\cap S_\ell^{++}\neq \emptyset$, then $\h$ contains a Berge-cycle of length $\ell-1$ and we are done in this case. 
 
Let $x$ and $y$ be two indices belonging to the set $S_\ell$. We define the relation $x\sim y$ as follows: if $x<y$, then $x\sim y$ if and only if the interval $(x,y]\cap S_1=\emptyset$; if $x>y$, then $x\sim y$ if and only if the interval $(y,x]\cap S_1=\emptyset$.
Clearly, $\sim$ is an equivalence relation. Assume $S_\ell$ has $m_1$, equivalence classes. Also, we say that a maximal subset of consecutive integers in $S_\ell$ is an interval of $S_\ell$. As $S^+_\ell\cap S_1 = \emptyset$ by the above,  elements of the same interval belong to the same equivalence class. Let $m_2$ be the number of intervals in $S_\ell$.  

If $\h$ does not contain cycles of length $\ell$ and $\ell-1$, then for the maximal element $z$ of each equivalence class, we have that $z+1,z+2\notin S_1$ and so by the definition of equivalence classes $z+1,z+2 \notin S_\ell$. Moreover, if an element $z'$ belongs both to $S_1$ and $S_\ell$, then $z'$ is the smallest element of an equivalence class. 
Also if $z''$ is the largest element of an interval that is not the rightmost interval in an equivalence class, then $z''+1\notin S_1\cup S_\ell$. 
In particular, we have $|S_1|,|S_{\ell}|\geq \floor{ \frac{k-3}{2}}$, $|\S_1\cap S_{\ell}|\leq m_1$, $|[\ell-2]\setminus (S_1\cup S_{\ell})|\geq 2(m_1-1)+(m_2-m_1)$  and $S_1\cup S_{\ell}\subseteq [\ell-2]$.
These observations show that $2\lfloor \frac{k-3}{2}\rfloor-m_1+2(m_1-1)+(m_2-m_1)\le \ell-2$ holds. 
As $\ell \le k$, we must have $m_2\le 4$.

Similarly, as in the proof of Claim~\ref{size_of_shadow} we can see that for any $j\in S_1^-$, the vertex $u_j$ is the endpoint of a longest path $\F_j$ with another end vertex $u_\ell$ and with defining vertex set $U$. Observe that the neighborhood $S_\ell$ of $u_\ell$ with respect to the non-defining hyperedges of $\F$ and $\F_j$ is the same, as the single hyperedge $h\in \F_j\setminus \F$ contains $u_1$ and therefore cannot contain $u_\ell$ without creating a cycle of length $\ell$. Therefore $S_{j,\F_j}\subseteq U$ and similarly as above if $[(S_{\ell}^{-}\cup S_{\ell}^{--})\cap (S_{\ell}^+~\cup~S_{\ell}^{++})]\cap ~S_{j,\F_j}\neq~\emptyset$, then $\h$ contains a Berge-cycle of length $\ell$ or $\ell-1$.

Let $S^*:=(S_{\ell}^{-}\cup S_{\ell}^{--})\cap (S_{\ell}^+\cup S_{\ell}^{++})$ then $|S^*|~\ge|S_\ell|-2m_2\ge\floor{\frac{k-3}{2}}-8$. Let $U_{S_1^-}:=\{u_i:i\in S_1^-\}$ and consider $E(U_{S_1^-})$.  Observe that all but one of the defining hyperedges of $\F_i$ are in $\F$, thus there are at most $|\F|+|U_{S_1^-}|\leq k-1+|S_1^-|$ hyperedges altogether in $E(U_{S_1^-})$ that are defining hyperedges of $\F$ or an $\F_i$. By the previous paragraph, all other hyperedges in $E(U_{S_1^-})$ are completely in $U\setminus S^*$, thus we have $$E(U_{S_1^-})\subseteq \binom{U\setminus S^*}{r}\cup \F \cup \bigcup_{x\in S_1^-}\F_x.$$ 
 
\noindent 
By the set degree condition and the above, we must have 
\begin{equation}\label{1}
|S_1^-|\binom{\floor{\frac{k-3}{2}}}{r-1}\le |E(U_{S_1^-})|\le \binom{k-\floor{\frac{k-3}{2}}+8}{r}+k-1+|S_1^-|.
\end{equation}
\noindent 
Using $\floor{\frac{k-3}{2}}\le |S_1^-|$, $\binom{a}{r}=\frac{a}{r}\binom{a-1}{r-1}$ and $\frac{\binom{a+1}{r-1}}{\binom{a}{r-1}}= \frac{a+1}{a-r+2}\le \frac{a}{a-r}$, and writing $k=\alpha r$ we have  \begin{equation}\label{2}
\begin{split}
    \binom{k-\floor{\frac{k-3}{2}}+8}{r} & =\frac{k-\floor{\frac{k-3}{2}}+8}{r}\binom{k-\floor{\frac{k-3}{2}}+7}{r-1} \\
    &
\le\left(\frac{k/2+9}{r}\right)\binom{\floor{\frac{k+17}{2}}}{r-1}=\left(\frac{\alpha}{2}+9/r\right)\binom{\floor{\frac{k+17}{2}}}{r-1}\\ & \le \left(\frac{\alpha}{2}+9/r\right)\left(\frac{\alpha}{\alpha-2}\right)^{10}\binom{\floor{\frac{k-3}{2}}}{r-1}.
\end{split}
\end{equation}

\noindent 
Therefore (\ref{1}), (\ref{2}) and $k-1+|S^-_1|\le 2k=2\alpha r$ implies $\alpha r/2-2\le (\frac{\alpha}{2}+9/r)(\frac{\alpha}{\alpha-2})^{10}+2\alpha$. This shows that for any $\varepsilon>0$, there is an $r_0$ such that if $r>r_0$, then $\alpha< 2+\varepsilon$ must hold, a contradiction. For the finitely many smaller values of $r$, the above inequality gives an upper bound $\beta_r$ for $\alpha=k/r$, which might be larger than $2+\varepsilon$. In that case we can choose $q_{\varepsilon}:=\max_{r\le r_0}\beta_r r$. Then we have $k> q_{\varepsilon} \ge \alpha r=k$, a contradiction.
\end{proof}

Note that the cycle $C$ given by Claim~\ref{Contains_cycle} is a longest Berge-cycle in $\h$ and let its defining vertices and defining hyperedges be $V:=\{u_1,u_2, \dots, u_{\ell-1}\}$ and $E(C):=\{e_1, e_2, \dots, e_{\ell-1}\}$, respectively, with $u_i,u_{i+1} \in e_i$. We have  $\ell$ is either $k-3, \ k-2,  \ k-1$ or $k$ by Claim~\ref{Contains_cycle}. Let us call $u_{i-1}$ and $u_{i+1}$ the \textit{neighbors of} $u_i$ \textit{on $C$}.

\subsection{Preliminary technical claims}

By Lemma~\ref{structure} (i), for any vertex $w\in V(\h)\setminus V$  we have $N_{\h \setminus C}(w)\subseteq V$. For any vertex $w\in V(\h)\setminus V$, we partition $N_{\h\setminus C}(w)$ into two parts the following way: let $M_w$ denote the set of vertices $v \in V$ such that there exists exactly one hyperedge in $\h \setminus C$ containing both $w$ and $v$, and let $D_w$ denote the set of those vertices $v\in V$ for which there exist at least 2 hyperedges in $\h \setminus C$ containing both $v$ and $w$. 

\begin{claim}\label{replace}
For any $w$ and $w'$ with $w,w' \in V(\h)\setminus V$ and not necessarily distinct, we have

(i) If $u_j \in N_{\h\setminus C}(w)$, $u_{j+1}\in N_{\h\setminus C}(w')$, then $w=w'$, $u_j,u_{j+1}\in M_w$ and there exists a non-defining hyperedge $h$ with $w,u_j,u_{j+1}\in h$.

(ii) If $u_j\in N_{\h\setminus C}(w)$, $u_{j+2}\in D_w$, then there exists a cycle $C'$ of length $\ell-1$ in $\h$ such that the defining vertices of $C'$ are those of $C$ but $u_{j+1}$ replaced by $w$.
\end{claim}
  \begin{figure}[ht] 
    \begin{center}
    \includegraphics[width=0.8\textwidth]{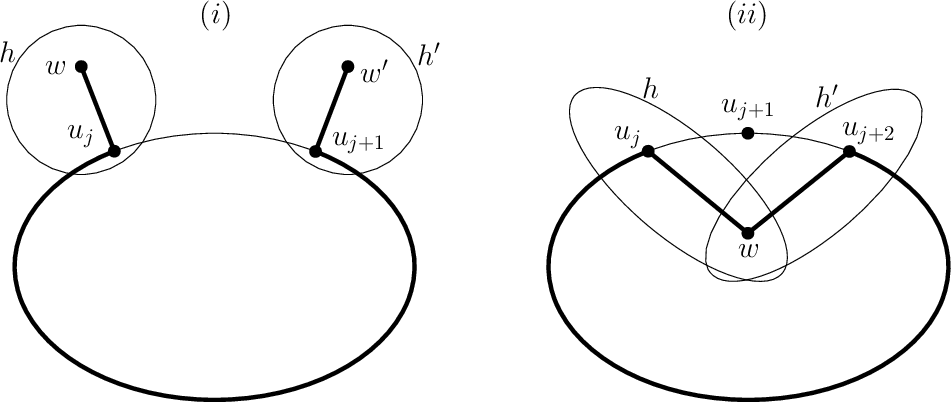}
    \caption{Sketch of the proof of Claim~\ref{replace}}
    \label{Figure:Replace}
    \end{center}
    \end{figure}

\begin{proof}   Let $u_j \in N_{\h\setminus C}(w)$, $u_{j+1}\in N_{\h\setminus C}(w')$. If $w\neq w'$, then for the hyperedges $h,h' \in \h \setminus C$ with $u_j,w\in h$ and $u_{j+1},w'\in h'$, we have $h\neq h'$, from Lemma~\ref{structure} (i). But then 
$$w',h',u_{j+1},e_{j+1},u_{j+2},\dots,u_{\ell-1},e_{\ell-1},u_1,e_1,\dots,u_j,h,w$$ is a Berge-path of length $\ell$, see Figure \ref{Figure:Replace}, a contradiction. So $w=w'$, and if there exist $h\neq h'$ with $u_j,w\in h$ and $u_{j+1},w\in h'$, then the Berge-path presented above is in fact a Berge-cycle that is longer than $C$, a contradiction. This proves (i).

For the second part of the claim, observe that if $u_j\in N_{\h\setminus C}(w)$ and $u_{j+2}\in D_w$, then there exist two distinct hyperedges $h, h' \in \h\setminus C$ such that $u_j,w\in h$ and $u_{j+2},w\in h'$, so in $C$ we can replace $e_j,u_{j+1},e_{j+1}$ by $h,w,h'$ to obtain the desired cycle $C'$, see Figure \ref{Figure:Replace}.
\end{proof}

\noindent 

\begin{claim}\label{replace2}
Suppose $u_{i-1},u_{i+1}, u_j\in D_w$ are three distinct vertices for some $w\in V(\h)\setminus V$ and let $w^*\in V(\h)\setminus V$ be a vertex distinct from $w$. Then we have the following.

(i) There is no hyperedge  $h\in \h\setminus C$ with $u_i,u_{j-1}\in h$ nor with $u_i,u_{j+1}\in h$.

(ii) If $u_{j+2}\in N_{\h\setminus C}(w)$, then $e_{i-1},e_{i}$ do not contain $u_{j+1}$.

(iii) Hyperedges $e_{i-1}$ and $e_i$ are not incident with the vertices  $w,w^*$.

(iv) Suppose $u_{t+1}\in D_{w^*}$ or $u_{t-1}\in D_{w^*}$ for some $t\neq i$. Then there is no  $h\in \h\setminus C$ incident to $u_i$ and $u_t$.

(v) The hyperedges $e_{j-1},e_j$ are not incident with $u_i$.
\end{claim}

  \begin{figure}[ht] 
    \begin{center}
    \includegraphics[width=1\textwidth]{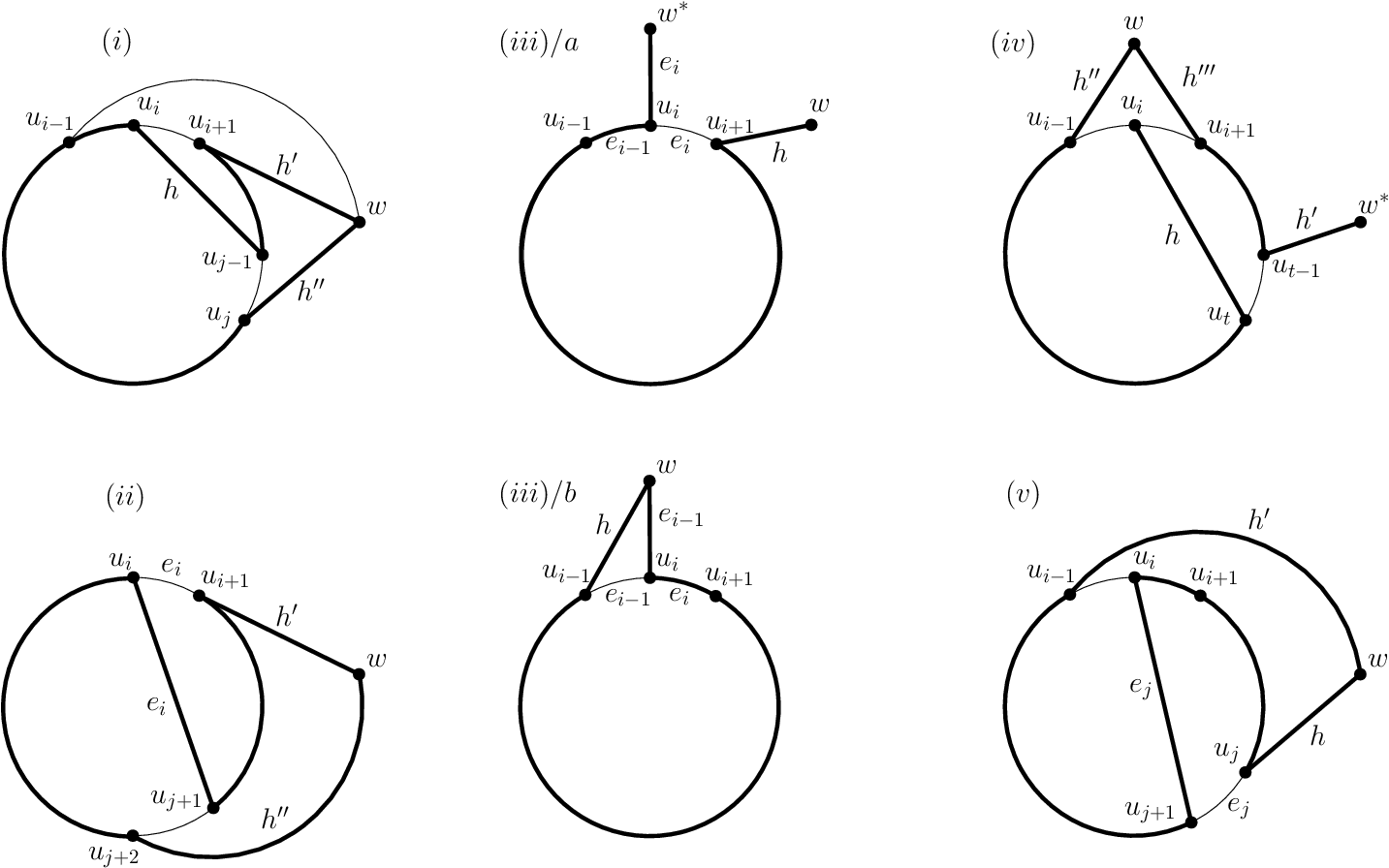}
    \caption{Sketch of the proof of Claim~\ref{replace2}}
    \label{Figure:Replace2}
    \end{center}
    \end{figure}

\begin{proof}
We start with the proof of $(i)$, see Figure \ref{Figure:Replace2} (i). Suppose by contradiction that $u_i,u_{j-1} \in h \in \h \setminus C$. Then by Claim~\ref{replace} (i), we have $w\notin h$ (as otherwise $u_i,u_{i-1}\in M_w$, contradicting $u_{i-1}\in D_w$). Furthermore, as $u_{i+1},u_j\in D_w$, there exist two distinct hyperedges $h',h'' \in \h \setminus C$ with $u_{i+1},w\in h'$ and $u_{j},w\in h''$. Using the fact that $u_{j-1}$ and $u_{i+1}$ are different vertices as there can not be neighboring vertices in $D_w$ by Lemma~\ref{structure} (ii), we have that  $$u_{i-1},e_{i-1},u_i,h,u_{j-1},e_{j-2},u_{j-2},\dots,u_{i+1},h',w,h'',u_j,e_j,u_{j+1},\dots,e_{i-2}$$ is a Berge-cycle longer than $C$, a contradiction. Similarly we can extend the cycle $C$ if $u_i,u_{j+1}~\in~h~\in~\h~\setminus~C$. This proves (i).

\smallskip 

To show (ii) see Figure \ref{Figure:Replace2} (ii), it is enough to get a contradiction if $e_i$ contains $u_{j+1}$, since the other case $e_{i-1}$ contains $u_{j+1}$ is symmetric. We have two non-defining distinct hyperedges, a hyperedge $h''$ incident to $w$ and $u_{j+2}$ and a hyperedge  $h'$ incident to $w$ and $u_{i+1}$ as $u_{i+1} \in D_w$. Then  $$u_i,e_i,u_{j+1},e_{j},u_j,e_{j-1},\dots,e_{i+1},u_{i+1},h',w,h'',u_{j+2},e_{j+2},\dots,u_{i-2},e_{i-2},u_{i-1},e_{i-1}$$ is a Berge-cycle longer than $C$, a contradiction. 

\smallskip 

To show statement (iii), suppose first $w^*\in e_i$. Then for a non-defining hyperedge $h$ incident to $w$ and $u_{i+1}$, we have that $w^*,e_i,u_{i},e_{i-1},u_{i-1},\dots,u_{i+1},h,w$ is a path of length $\ell$ - a contradiction. If $w^*\in e_{i-1}$, then similarly, for a non-defining hyperedge $h$ incident to $w$ and $u_{i-1}$, we have that $w^*,e_{i-1},u_{i},e_i,u_{i+1},\dots,e_{i-2},u_{i-1},h,w$ is a path of length $\ell$ - a contradiction.
If $w \in e_{i-1}$, then we have a contradiction since there exists a cycle longer than $C$, which is obtained from $C$ by exchanging the edge $e_{i-1}$ with $h,w,e_{i-1}$, where $h$ is a non-defining hyperedge incident to $w$ and $u_{i-1}$. Similarly, we get a contradiction if  $w \in e_{i}$. 

\smallskip 

To prove (iv) by a contradiction, suppose that we have a non-defining hyperedge $h$ of $C$ incident to $u_i$ and $u_t$. Assume  without loss of generality that $u_{t-1}\in D_{w^*}$ since the other case is symmetrical. Then there exists a non-defining hyperedge $h'$ different from $h$, incident to  $u_{t-1}$ and $w^*$. Also there are two distinct  non-defining hyperedges $h'', h'''$ with $w, u_{i-1} \in h''$  and $w, u_{i+1} \in h'''$. At first note that hyperedge $h$ is distinct from $h''$ and $h'''$ by Claim~\ref{replace} (i). From Lemma~\ref{structure} (i) we have that hyperedges $h''$ and $h'''$ distinct from $h'$. Finally, we have a contradiction since the following  is a Berge-path of length $\ell$ 
\[
w^*,h',u_{t-1}, e_{t-2}, \cdots u_{i+1}, h''', w, h'', u_{i-1}, e_{i-2}, \cdots,  u_{t+1},e_{t}, u_t, h, u_i.\]

To prove (v) suppose by a contradiction that $e_j$ contains $u_i$. There are distinct non-defining hyperedges $h,h'$ with $w,u_j\in h$ and $w,u_{i-1}\in h'$. Then $$u_{j+1},e_j,u_i,e_i,u_{i+1},e_{i+1},\dots,e_{j-1},u_j,h,w,h',u_{i-1},e_{i-2},u_{i-2},\dots, e_{j+2},u_{j+2},e_{j+1}$$ is a Berge-cycle of length longer than $C$. This contradiction proves (v). The proof for the case $u_i\in e_{j-1}$ is analogous.
\end{proof}

By Claim~\ref{replace} (i) and the set degree condition

$$\binom{\lfloor\frac{k-3}{2}\rfloor}{r-1} \le |M_w|+ \binom{\min \left\{{\floor{\frac{\ell-1-|M_w|}{2}}}, |D_w| \right\}}{r-1}$$

\noindent 
must hold for all $w\in V(\h)\setminus V(C)$. At first we observe that $|M_w| \le 3$ as otherwise $\ell-1-|M_w|\le k-5$, thus the above inequality implies $\binom{\lfloor\frac{k-3}{2}\rfloor}{r-1} \le \ell-1+\binom{\lfloor\frac{k-5}{2}\rfloor}{r-1}$, a contradiction if $k$ is sufficiently large. Therefore, we have  $\binom{\lfloor\frac{k-3}{2}\rfloor}{r-1} \le 3+\binom{|D_w|}{r-1}$, thus for $k\ge 11$ we have 
\begin{equation}\label{dewe}
    |D_w|\ge \lfloor \frac{k-3}{2}\rfloor.
\end{equation}

\medskip 

Recall that $V=\{u_1,\dots,u_{\ell-1}\}$ is the set of defining vertices of $C$, and $V$ differs from $V(C)$, which is the set of all the vertices contained in the defining hyperedges of $C$.
We say that a vertex $u_i \in V$ is \textit{replaceable by $w$}, if  $u_{i-1},u_{i+1} \in D_w$, and we denote by $R_w$ the set of vertices that are replaceable by $w$. This name comes from the fact that we can obtain another cycle $C'$ of length $\ell-1$ from $C$ by replacing $u_i$ by $w$, replacing $e_{i-1}$ by a hyperedge $h$ from $D_w$ containing $u_{i-1}$, and replacing $e_i$ by a hyperedge $h'$ from $D_w$ containing $u_{i+1}$ such that $h'\neq h$.

A vertex is called \textit{replaceable}, if it is replaceable by $w$ for some $w\in V(\h)\setminus V$. For a replaceable vertex $w'=u_i$, we define $D_{w'}$ and $M_{w'}$ as for vertices in $V(\h)\setminus V$, except $C$ is replaced by $C'$.

For a vertex $w \in V(\h) \setminus V$ let us call a maximal set $I$ of consecutive defining vertices of $C$ in $V \setminus D_w$ a \textit{missing interval for $w$} (or just missing intervals, if $w$ is clear from the context) if its size is at least two. Let $I_1,I_2,\dots,I_s$ be the missing intervals of $C$ for $w$ and let us denote by $\overline{I_1}, \overline{I_2}, \ldots, \overline{I_s}$ the same intervals without the terminal vertices (it is possible that $\overline{I_j}=\emptyset$). We have $\sum_{i=1}^s(|I_i|-1)=\ell-1-2|D_w|$ by Claim~\ref{replace} (i). In particular, as $|D_w|\ge \floor{\frac{k-3}{2}}$ by (\ref{dewe}),
%the set degree condition and Lemma~\ref{structure} (i), 
we have $s\le 3$, if $k$ is even and $s\le 2$, if $k$ is odd. 
Let us consider a hyperedge $e_j\in C$ such that $u_j$ or $u_{j+1}$ is from a missing interval.
The number of such hyperedges is $\sum_{i=1}^s(|I_i|+1)$, 

\begin{observation}\label{obsi}
    $\sum_{i=1}^s(|I_i|+1)$ is at most 9, if $k$ is even, and at most 6, if $k$ is odd.
\end{observation}
%which is at most 9, if $k$ is even, and at most 6, if $k$ is odd.
Our next technical claim is about missing intervals.

\begin{claim}\label{replace3}
Suppose that $u_i,u_{i+1},\dots,u_{i+t}$ form a missing interval for some $w\in V(\h)\setminus V$. Then 

(i) $e_{i-1}$ and $e_{i+t}$ do not contain vertices 
%$w^*\in V(\h)\setminus V$
outside $V$; and 

(ii) if $u_{i-1}\in D_{w'}$ (resp. $u_{i+t+1}\in D_{w'}$) for some $w'\neq w$, then $e_{i-1}$ (resp. $e_{i+t}$) does not contain a vertex from $R_w $. 
\end{claim}

\begin{proof}
To prove (i) observe that there exists a Berge-path starting with the vertex $w$, a non-defining hyperedge $h$, the vertex $u_{i-1}$, going around $C$ with defining vertices and hyperedges and finishing with the vertex $u_i$. Such $h$ exists since $u_{i-1}$ does not belong to the missing interval, so $u_{i-1}\in D_w$. Note that we did not use the hyperedge $e_{i-1}$. Assume that $e_{i-1}$ contains $w^*\in V(\h)\setminus V$. If $w=w^*$, then $e_{i-1}$ closes a Berge-cycle longer than $C$, a contradiction, while if $w\neq w^*$, then finishing with $e_{i-1},w^*$ we obtain a Berge-path of length $\ell$, a contradiction. This contradiction proves (i). A similar argument shows the statement for the hyperedge $e_{i+t}$. 

%We omit the proof of part $(ii)$ since the same argument will provide the desired result after replacing a replaceable vertex with $w$.

To prove (ii), assume that $u_j\in R_w$ is in $e_{i-1}$ and apply Claim~\ref{replace} (ii) to obtain a cycle $C'$ by replacing $u_j$ by $w$. In $C'$, we have that $e_{i-1}$ does not contain vertices outside the defining vertices of $C'$, as in the proof of (i). But $e_{i-1}$ contains $u_j$, which is not a defining vertex of $C'$, a contradiction.
\end{proof}

Here we will show that $|D_{w^*}|\ge \floor{\frac{k-3}{2}}$ holds even for vertices $w^*\in V(C)\setminus V$, therefore we have $|D_{w'}|\ge \floor{\frac{k-3}{2}}$ for all $w' \in V(\h) \setminus V$.  

By Claim~\ref{replace2} (iii) and Claim~\ref{replace3} (i), if $w^*\in V(C)\setminus V$ and $u_i\in D_w$, then $w^*\notin e_{i-1},e_i$.
Therefore the number of defining hyperedges that may contain $w^*$ is at most 3. So Claim~\ref{replace} and the set degree condition implies
$$\binom{\lfloor\frac{k-3}{2}\rfloor}{r-1} \le 3+|M_{w^*}|+ \binom{\min \left\{{\floor{\frac{\ell-1-|M_{w^*}|}{2}}}, |D_{w^*}| \right\}}{r-1}.$$
Just as for $w\in V(\h)\setminus V(C)$, in two steps we obtain $|D_w|\ge \lfloor\frac{k-3}{2}\rfloor$ for $k$ large enough.

\medskip 

Before continuing with giving possible embeddings of $\h$ into some $\h_{n,a,b_1,...,b_s}$ let us state a last technical claim that will be used several times. Let us recall that a terminal vertex $v$ is the vertex of a missing interval that is adjacent to a vertex from $D_w$.

\begin{claim}\label{2intervals}
Suppose $D_w=D_{w'}$ for some  $w'\in V(\h)\setminus V$ with $w'\neq w$.

\smallskip 

(i) There does not exist $h\in \h\setminus C$ such that $h$ contains terminal vertices of two distinct missing intervals of $w$.

\smallskip 
(ii) If $\{ u_i,u_{i+1},u_{i+2} \}$ and $ \{ u_j,u_{j+1} \}$ form missing intervals of $w$ and there exists $h\in \h\setminus C$ with $u_{i+1},u_j\in h$ or $u_{i+1},u_{j+1}\in h$, then there does not exist $h'\in \h\setminus C$, with $u_i,u_{i+2}\in h'$.
\end{claim}

\begin{proof}

  We prove (i) by contradiction. Suppose $\{ u_i,u_{i+1},\dots,u_{i+t} \}$ and $\{ u_j,u_{j+1},\dots,u_{j+z} \}$ are two distinct missing intervals of $w$.
  
  $\bullet$ Suppose first $u_i,u_{j+z}\in h\in \h\setminus C$. We have $ u_{i+t+1}, u_{i-1}, u_{j+z+1}\in D_w$, therefore there are three different hyperedges $h_w$, $h_w'$ and $h_{w'}$, such that $h_w$ is incident to $w$ and $u_{i+t+1}$, $h_w'$ is incident to $w$ and $u_{i-1}$ and  $h_{w'}$ is incident to  $u_{j+z+1}$ and $w'$. Note that all those hyperedges are different from $h$ by Claim \ref{replace} (i). Then we have a contradiction since the following Berge-path is of length $\ell$, as it contains all the $\ell-1$ defining vertices of $C$ and $w$ and $w'$:
 \[
 u_{i+t}, \dots, u_i, h, u_{j+z}, e_{j+z-1}, \dots, u_{i+t+1}, h_w, w, h_w', u_{i-1}, e_{i-2},\dots, u_{j+z+1}, h_{w'}, w'.
\]

 $\bullet$ If $u_{i+t},u_{j+z}\in h\in \h\setminus C$, then the Berge-path of length $\ell$ (using similar ideas as in the previous bullet) is
  \[
 u_{i}, \dots, u_{i+t}, h, u_{j+z}, e_{j+z-1}, \dots, u_{i+t+1}, h_w, w, h_w', u_{i-1}, e_{i-2},\dots, u_{j+z+1}, h_{w'}, w',
\]
 
and we are done with the proof of (i).

\medskip 

In (ii) we can assume that $u_{i+1},u_{j+1}\in h$ holds since the case $u_{i+1},u_{j}\in h$ is identical. The proof of this part is similar, at first we observe from part (i) that we have $h\neq h'$. Then the following Berge-path of length $\ell$ gives us a contradiction: 
\[
u_i,h',u_{i+2},e_{i+1},u_{i+1},h,u_{j+1},e_j,u_{j},e_{j-1},\dots, u_{i+3}, h_w, w, h_w', u_{i-1}, e_{i-2},\dots, u_{j+2}, h_{w'}, w'.
\]
\end{proof}

\subsection{Possible embeddings of $\h$}

We are now able to give possible embeddings of $\h$ into some $\h_{n,a,b_1,...,b_s}$. In this subsection, we gather all the information that we know about these embeddings so far and in the next subsection, we analyze further the different cases to finish the proof.

Let us fix $w\in V(\h)\setminus V$ with $D_w$ of maximum size and let $\h^*$ denote the sub-hypergraph of $\h$ that we obtain by removing those defining hyperedges $e_i$ of $C$ for which at least one of $u_i$ or $u_{i+1}$ is a vertex of a missing interval for $w$. By Observation \ref{obsi}, $|\h|\le |\h^*|+9$. 

\medskip 

 If we are in a case when for all $w' \in V(\h) \setminus V$ we have $D_{w'} \subseteq D_w$, then let $A=D_w$, $B_i=I_i$ for $i=1,2,\ldots,s$ and $L=V(\h)\setminus \left(D_w\cup (\cup_{i=1}^sI_i)\right)$. Let us summarize the findings of the technical claims and enumerate the types of different hyperedges in 
$\h\setminus \h_{n,a,b_1,b_2,\dots,b_s}$ in this scenario.

\begin{summary}\label{sum1} \ 

\noindent Assume that for all $w' \in V(\h) \setminus V$ we have $D_{w'} \subseteq D_w$.
If $h\in \h\setminus \h_{n,a,b_1,b_2,\dots,b_s}$ is not a defining hyperedge of $C$ (i.e., $h \in \h \setminus C$), then
\begin{enumerate}

    \item 
    either there exists $v\in (V(\h)\setminus V)\cup R_w$ such that $h\setminus \{v\} \subseteq D_w \cup \bigcup_{i=1}^s\overline{I_i}$ and $h\cap \bigcup_{i=1}^s\overline{I_i}\neq \emptyset$ (we refer to these hyperedges as type 1 hyperedges in what follows); or
    \item
    $h\subseteq V\setminus R_w$ and $h$ contain vertices from at least two distinct missing intervals. We refer to these hyperedges as type 2 hyperedges in what follows.
\end{enumerate}
If $e_i\in \h\setminus \h_{n,a,b_1,b_2,\dots,b_s}$ is a defining hyperedge of $C$, then
\begin{enumerate}[resume]
    \item 
    either $e_i\in \h\setminus \h^*$; or
    \item
    $u_i$ or $u_{i+1}$ belongs to $R_w$, $e_i\setminus \{u_i,u_{i+1}\}\subseteq  D_w \cup \bigcup_{i=1}^sI_i$ and $e_i \cap \bigcup_{i=1}^sI_i\neq \emptyset$.
\end{enumerate}
\end{summary}

\begin{proof} 

Suppose first that $h$ is not a defining hyperedge of $C$ and $h$ contains a vertex $v\in (V(\h)\setminus V) \cup R_w$. We claim that $h$ cannot contain any $v'\in V(\h)\setminus V$ with $v'\neq v$. Indeed, if $v\notin V$, then it follows from Lemma~\ref{structure} (i). If $v\in R_w$ and $v'=w$, then $w$ can be inserted to obtain a longer cycle than $C$, while if $w\neq v'$, then using $h$ and the defining vertices and hyperedges of $C$ one can create a Berge-path of length $\ell$ from $v'$ to $w$.

We also claim that $h$ cannot contain a neighbor of a vertex in $D_w$ on $C$. Indeed, assume that $v\notin V$. If $h$ contains $u_{i-1}$ and $u_i\in D_w$, then a hyperedge in $\mathcal{H}\setminus \mathcal{C}$ that differs from $h$ and contains both $w$ and $u_i$. This contradicts Lemma~\ref{structure} (ii).
%if $v\notin V$, then it follows from Lemma~\ref{structure} (ii) and (iii). 
If $v\in R_w$, then $h$ cannot contain a neighbor of a vertex in $D_w$ on $C$ by Claim~\ref{replace2} (i).
Therefore, $h$ cannot contain other vertices of $R_w$, nor terminal vertices of missing intervals. This gives possibility 1, in other words shows that $h$ is of type 1.
 
\smallskip  
 
 Otherwise if $h\in \h\setminus\h_{n,a,b_1,b_2,\dots,b_s}$ is not a defining hyperedge of $C$, then we must have $h\subseteq V\setminus R_w$. As all hyperedges in $\binom{A\cup I_j}{r}$ belong to $\h_{n,a,b_1,b_2,\dots,b_s}$, there must exist two distinct missing intervals meeting $h$.  This gives possibility 2, in other words shows that $h$ is of type 1.
 
\medskip 
    
Let $e_i\in \h\setminus \h_{n,a,b_1,b_2,\dots,b_s}$ be a defining hyperedge of $C$. If at least one of $u_i$ or $u_{i+1}$ belongs to a missing interval, then $e_i \in \h\setminus \h^*$ by definition of $\h^*$. This gives possibility 3. Note that we have more information on some of these hyperedges by Claim~\ref{replace3}.

\smallskip

Otherwise $u_i$ or $u_{i+1}$ belongs to $R_w$. By Claim~\ref{replace2} (ii), $e_i$ does not contain any other vertex from $R_w$, and by Claim~\ref{replace2} (iii)  $e_i$ cannot contain any vertex from $V(\h)\setminus V$. 
This gives us possibility 4. Even more, if the unique element of $e_i\cap R_w$ is also replaceable by some $w'\neq w$, then $e_i$ cannot contain $w$ either. 
\end{proof}

If we are in a case when we have vertices $w,w'\in V(\h)\setminus V$ with $D_w\not\subseteq D_{w'}$ and $D_{w'}\not\subseteq D_w$, then as $\floor{\frac{k-3}{2}}\le |D_w|,|D_{w'}|$, we will have $\floor{\frac{k-1}{2}}\le |D_w\cup D_{w'}|$. Since the elements of $D_w\cup D_{w'}$ cannot be neighbors on $C$ by Claim \ref{replace} (i) and $|C|\le k-1$, we must have $|D_w\cup D_{w'}|=\floor{\frac{k-1}{2}}$. 

If $|C|=2\floor{\frac{k-3}{2}}+2$, then we will embed $\h$ to $\h_{n,\floor{\frac{k-1}{2}}}$, with $A=D_w\cup D_{w'}$ and all the other vertices are going to $L$.

Recall that $\h^*$ denotes the sub-hypergraph of $\h$ that we obtain by removing those defining hyperedges $e_i$ of $C$ for which at least one of $u_i$ or $u_{i+1}$ is a vertex of a missing interval for $w$.
If $|C|=2\floor{\frac{k-3}{2}}+3$, then we will embed $\h^*$ to $\h_{n,\floor{\frac{k-1}{2}},2}$  with $A=D_w\cup D_{w'}$, the unique missing interval goes to $B_1$ and all the remaining vertices are going to $L$.

\begin{summary}\label{sum2} \ 
If for $w,w'\in V(\h)\setminus V$ we have $D_w\not\subseteq D_{w'}$ and $D_{w'}\not\subseteq D_w$, then

\begin{enumerate}
    \item 
    there is no hyperedge  $h\in \h\setminus C$ with $h\in \h\setminus \h_{n,\floor{\frac{k-1}{2}}}$ or $h\in \h\setminus \h_{n,\floor{\frac{k-1}{2}},2}$ depending on whether $|V|=2\floor{\frac{k-3}{2}}+2$ or $|V|=2\floor{\frac{k-3}{2}}+3$; and
    \item
    if $u_{i-1},u_{i+1}\in D_w\cup D_{w'}$, then $e_{i-1}\setminus \{u_i\},e_i\setminus \{u_i\}\subseteq D_w\cup D_{w'}\cup I$, where $I$ is the unique possible interval $u_j,u_{j+1}$ of size two disjoint with $D_w\cup D_{w'}$. Furthermore, if $u_i$ is replaceable by either $w$ or $w'$, then $e_{i-1}\setminus \{u_i\},e_i\setminus \{u_i\}\subseteq D_w\cup D_{w'}$.
\end{enumerate}
\end{summary}

\begin{proof}
Note that every $u\in V\setminus (D_w\cup D_{w'})$ has a neighbor on $C$ in $D_w\cup D_{w'}$. Therefore, if $v\in h\in \h\setminus C$ with $v\in V(\h)\setminus V$, then Claim~\ref{replace} (i) yields $h\setminus \{v\}\subseteq D_w\cup D_{w'}$. So we only have to consider hyperedges $h\subset V$. If $u_i$ is replaceable by either $w$ or $w'$ and $u_i\in h\in \h\setminus C$, then Claim~\ref{replace2} (i) and (iv) yield $h\setminus \{u_i\}\subseteq D_w\cup D_{w'}$.
%the same. 
Finally, if $u_j,u_{j+1}$ form the unique interval of $V\setminus (D_w\cup D_{w'})$, and $u_i$ is neither replaceable by $w$ nor by $w'$, then one of $u_{i-1},u_{i+1}$ belong to $D_w$, the other to $D_{w'}$.  Suppose that $u_i,u_j\in h\in \h\setminus C$, the other case $u_i,u_{j+1}\in h\in \h\setminus C$ is symmetric. Then $u_{j-1}\in D_{w^*}$ and $u_{i-1}\in D_{w^{**}}$ for some $w^*,w^{**}\in \{w,w'\}$. Therefore $$w^*,h',u_{j-1},e_{j-2},\dots,u_{i+1},e_{i},u_i,h,u_j,e_j,u_{j+1},\dots,e_{i-2},u_{i-1},h'',w^{**}$$ is either a cycle (if $w^*=w^{**}$) or a path (if $w^*\neq w^{**}$)
of length $k$. Such distinct hyperedges $h',h''$ exist from the definition of $D_{w^*}, D_{w^{**}}$  as well as they are different from the hyperedge $h$ since $h \subset V$. This settles part 1.

For part 2, let us consider defining hyperedges $e_{i-1},e_i$ of $C$ with $u_{i-1},u_{i+1}\in D_w\cup D_{w'}$. Observe first that all but at most one of the $u_i$'s are replaceable either by $w$ or by $w'$. If $u_i$ is indeed replaceable by $w$ or by $w'$, then  Claim~\ref{replace2} (iii) yields $e_{i-1}\setminus \{u_i\},e_i\setminus \{u_i\}\subseteq D_w \cup D_{w'}$. For the at most one exception $u_i$, we have that one of $u_{i-1},u_{i+1}$ is in $D_w$, the other one is in $D_{w'}$ and by Claim~\ref{replace2} (v) we are done.
\end{proof}

\subsection{Case-by-case analysis}

We finish the proof with a case-by-case analysis according to the  length of the longest Berge-cycle  $C$ and subcases will be defined according to the size of $D_w$. Let us remind the reader that the length of the cycle $C$, $\ell-1$, might take the values $2\floor{\frac{k-3}{2}}$, $2\floor{\frac{k-3}{2}}+1$, $2\floor{\frac{k-3}{2}}+2$ or $2\floor{\frac{k-3}{2}}+3$, and in the last case $k$ is even. In each case, we will use the summaries from the previous subsection.

\medskip 

\textsc{\textbf{Case I}} $\ell-1=2\floor{\frac{k-3}{2}}$.

\smallskip

As $|D_w|\ge \floor{\frac{k-3}{2}}$, then by Claim~\ref{replace} (i), $D_w$ must consist of every second vertex of $V$, so there are no missing intervals.
Summary~\ref{sum1} implies $\h \subseteq \h_{n,\floor{\frac{k-3}{2}}}$ thus $$|\h|\le |\h_{n,\floor{\frac{k-3}{2}}}|<|\h_{n,\floor{\frac{k-3}{2}},3}|,$$ which contradicts the assumption on $|\h|$.

\medskip

\textsc{\textbf{Case II}} $\ell-1=2\floor{\frac{k-3}{2}}+1$.

\smallskip

$|D_w|\ge \floor{\frac{k-3}{2}}$ and Claim~\ref{replace} (i) imply that after a possible relabelling we have  $D_w=\{u_1,u_4,\dots,u_{2\floor{\frac{k-3}{2}}}\}$ and thus $\{u_{2},u_3\}$ is the only missing interval for $w$, and all other vertices in $V\setminus D_w$ are in $R_w$. As all vertices in $V\setminus D_w$ are neighbors to some vertex in $D_w$, by Summary~\ref{sum1}, all  hyperedges in $\h \setminus C$ belong to $\h_{n,\floor{\frac{k-3}{2}},2}$. 

To consider the defining hyperedges of $C$, let us analyze those that contain an $u_i \in R_w$. Observe that by Claim \ref{replace} (i)  a vertex in $D_w$ cannot be a neighbor on $C$ of a vertex in $D_{w'}$ for some $w' \in V(\h) \setminus V$ with $w'\neq w$. This implies that $D_{w'}\subseteq\{u_1,u_4,\dots, u_{2\floor{\frac{k-3}{2}}}\}$. Since $|D_{w'}|\ge \floor{\frac{k-3}{2}}$, we have $D_w=D_{w'}$ for any two $w,w'\in V(\h) \setminus V$. We apply Claim~\ref{replace} (ii) to replace $u_i$ by $w$. Then $D_{u_i}$ becomes equal to $D_{w'}$, thus contains $u_1$ and $u_4$. By Claim~\ref{replace} (i), $u_2$ and $u_3$ are not in $N_{\h\setminus C'}(u_i)$, in particular they are not in $e_{i-1}$, nor in $e_i$.
%Hence we have that $e_{i-1}$ and $e_i$ cannot contain any of $u_{2}$ and $u_3$ by Claim~\ref{replace} (i) applied to the cycle $C'$ we obtain by Claim~\ref{replace} (ii). 

We obtained that $e_{i-1}$ and $e_i$ cannot contain any of $u_{2}$ and $u_3$.
Therefore, in Summary~\ref{sum1} possibility 3 holds, hence we have that $\h^*\subseteq \h_{n,\floor{\frac{k-3}{2}},2}$, thus $$|\h|\le |\h_{n,\floor{\frac{k-3}{2}},2}|+3< |\h_{n,\floor{\frac{k-3}{2}},3}|,$$ which contradicts the assumption on $|\h|$.

\medskip

\textsc{\textbf{Case III}} $\ell-1=2\floor{\frac{k-3}{2}}+2$.

\smallskip

The three subcases below cover this case.

\medskip

\textsc{\textbf{Case III/A}} There exists $w\in V(\h)\setminus V$ with $|D_w|=\floor{\frac{k-3}{2}}+1$.

\smallskip

Then there is no missing interval for $w$, and so $V\setminus D_w \subseteq R_w$, so by Summary \ref{sum1} we have $\h=\h^*\subseteq \h_{n,\floor{\frac{k-1}{2}}}$. 

\medskip

\textsc{\textbf{Case III/B}} There exists $w\in V(\h)\setminus V$, for which there are  two missing intervals, $\{u_i,u_{i+1}\}$ and $\{u_j,u_{j+1}\}$.

\medskip 

Note that there is no type 1 hyperedge of $\h\setminus C$, as each vertex of the missing intervals is terminal. Observe that all the vertices in $V\setminus D_w$ have neighbors in $D_w$, therefore the fact that $|D_{w'}|\ge \floor{\frac{k-3}{2}}$, together with $|D_{w}|= \floor{\frac{k-3}{2}}$ and Claim~\ref{replace}~(i) imply $D_w=D_{w'}$ for all $w,w'\in V(\h)\setminus V$.
This enables us to conclude that

\smallskip 

- by Claim~\ref{2intervals} (i), there is no hyperedge $h\in \h\setminus C$ of type 2; and

\smallskip 

- by Claim~\ref{replace2} (v), if $u_l\in R_w$, then $e_{l-1},e_l$ do not contain vertices of missing intervals. 

\medskip  

So by Summary \ref{sum1} we have $\h^*\subseteq \h_{n,\floor{\frac{k-3}{2}},2,2}$ and thus $$|\h|\le |\h_{n,\floor{\frac{k-3}{2}},2,2}|+6<|\h_{n,\floor{\frac{k-3}{2}},3}|,$$ contradicting the assumption on $|\h|$.

\medskip

\textsc{\textbf{Case III/C}} For all $w\in V(\h)\setminus V$, there is only one missing interval containing three vertices $\{u_{i(w)},u_{i(w)+1},u_{i(w)+2}\}$.

\smallskip

If there exist two vertices $w,w' \in V(\h) \setminus V$ with $i(w)\neq i(w')$, then $D_w\cup D_{w'}$ must contain every second vertex of $C$. So by Summary \ref{sum2}, we have $\h\subseteq \h_{n,\floor{\frac{k-1}{2}}}$ as claimed by the theorem.

So we can assume that the $D_w$s are the same and without loss of generality suppose that for every  $w\in V(\h)\setminus V$, the missing interval is $\{u_1,u_2,u_3\}$. Moreover, as every replaceable vertex $u_i$ is replaceable by any $w\in V(\h)\setminus V$, replaceable vertices and defining hyperedges $e_{i-1},e_i$ behave as vertices in $V(\h)\setminus V$ and hyperedges in $\h\setminus C$. By Summary \ref{sum1} and the above, we have to deal with type 1 hyperedges of $\h \setminus C$ and $\h\setminus \h^*=\{e_{2\floor{\frac{k-3}{2}}+2},e_1,e_2,e_3\}$.

$\bullet$ At first suppose that there exists a type 1 hyperedge of $\h \setminus C$, i.e., $h\in \h\setminus C$ with $v,u_2\in h$ for some $v\in (V(\h)\setminus V)\cup R_w$. Without loss of generality, we may assume $v\in V(\h)\setminus V$. Then we claim that there is no hyperedge $h'\in \h$ with $u_1,u_3\in h'$. Suppose by a contradiction that such $h'$ exists, then observe that $h'\neq h$, as otherwise, we would have $v,u_1,u_3\in h'$ that is not possible by Summary~\ref{sum1}. Also, either $h'\notin \{e_1,e_3\}$ or $h'\notin \{e_{2\floor{\frac{k-3}{2}}+2},e_2\}$, so we may assume $h'\notin \{e_1,e_3\}$ without loss of generality. 
Since $u_{2\floor{\frac{k-3}{2}}+2}\in D_v$, there is a hyperedge $h''$ different from the hyperedges $h$ and $h'$, incident to the vertices $v$ and $u_{2\floor{\frac{k-3}{2}}+2}$. We have a contradiction since the following is a longer Berge-cycle than $C$, containing all defining vertices of $C$ and $v$: $$v,h,u_2,e_1,u_1,h',u_3,e_3,u_4,\cdots,u_{2\floor{\frac{k-3}{2}}+2},h''. $$
As no hyperedge contains both $u_1$ and $u_3$, we obtained $\h\subseteq \h_{n,\floor{\frac{k-1}{2}}}$ in this case. 

$\bullet$ Suppose next that there is no type 1 hyperedge of $\h \setminus C$, i.e., by Summary \ref{sum1}, we have $\h^*\subseteq \h_{n,\floor{\frac{k-3}{2}},3}$. Observe that $e_{2\floor{\frac{k-3}{2}}+2}$ and $e_3$  do not contain vertices from $(V(\h)\setminus V) \cup R_w$ by Claim~\ref{replace3} (i) and (ii). If the same holds for $e_1,e_2$, then $\h\subseteq \h_{n,2\floor{\frac{k-3}{2}},3}$ contradicting the assumption $|\h|>| \h_{n,2\floor{\frac{k-3}{2}},3}|$. So we can assume that $e_2$ contains a vertex $v \in (V(\h)\setminus V) \cup R_w$. Then we claim that there is no hyperedge $h\in \h\setminus C$ with $u_1,u_3\in h$. Here we get a contradiction as in the previous settings with a longer Berge-cycle, therefore we omit the proof. 
We obtained the following contradiction $$|\h|\le 2+|\h_{n,\floor{\frac{k-3}{2}},3}|-\binom{\floor{\frac{k-3}{2}}}{r-2}<|\h_{n,\floor{\frac{k-3}{2}},3}|.$$

\medskip

\textsc{\textbf{Case IV}} $\ell-1=2\floor{\frac{k-3}{2}}+3$.

\smallskip

Note that in this case $k$ is even and the length of $C$ is $k-1$. We again distinguish several subcases.

\textsc{\textbf{Case IV/A}} 
 $|D_w|=\floor{\frac{k-1}{2}}=\frac{k}{2}-1$.

\smallskip

Then as $D_w$ does not contain neighboring vertices on $C$, after relabelling, we can suppose that we have $D_w=\{u_1,u_4,u_6,\dots,u_{k-2}\}$. So there is one missing interval $\{u_2,u_3\}$, therefore there does not exist a type 1 or type 2 hyperedge $h\in\h\setminus C$. If $u_i\in R_w$, then by Claim~\ref{replace2}~(iii) $e_{i-1}$ and $e_i$  do not contain vertices from $V(\h)\setminus V$. We claim that $e_{i-1}$ and $e_i$ do not contain vertices from the missing interval $\{u_2,u_3\}$. Indeed, if there exists $w^*\neq w$ with $u_1\in N_{\h\setminus C}(w^*)$ and $u_2\in e_i$ or $e_{i-1}$, then the following is a Berge-path of length $k$: $$u_i, e_i \textrm{ (or } e_{i-1} \textrm{), } u_2,e_2,u_3,\dots, u_{i-1},h,w,h',u_{i+1},e_{i+1},\dots,u_{k-1},e_{k-1},u_1,h'',w^*.$$ Here $h$ and $h'$ exist and are distinct as $u_i$ is in $R_w$ and $h''$ exists by the choice of $w^*$.

Similarly, if there exists $w^{**}\neq w$ with $u_4\in N_{\h\setminus C}(w^{*})$, then $e_{i-1},e_i$ cannot contain $u_3$. As all $D_{w^*}$ is of size at least $\floor{\frac{k-3}{2}}$, the only cases when we are not yet done is when $u_1\notin N_{\h\setminus C}(w^*)$ and $D_{w^*}=\{u_4,u_6,\dots,u_{2\floor{\frac{k-3}{2}}+2}\}$ or $u_4\notin N_{\h\setminus C}(w^*)$ and $D_{w^*}=\{u_6,u_8,\dots,u_{2\floor{\frac{k-3}{2}}+2},u_1\}$. By symmetry, we can assume the first. But then any replaceable $u_i$ but $u_{2\floor{\frac{k-3}{2}}+3}$ can be replaced with some $w^*\neq w$, and the above arguments  applied to the new cycle $C'$ show that any $u_i\in h\in \h\setminus C'$ (in particular, it applies to $e_i$ and $e_{i-1}$!) cannot contain $u_3$, and by Summary~\ref{sum1}, we already know that $e_{i-1},e_i$ cannot contain $u_{2\floor{\frac{k-3}{2}}+3}$. Therefore setting $A=D_w\setminus \{u_1\}$, $B_1=\{u_{2\floor{\frac{k-3}{2}}+3},u_1,u_2,u_3\}$ we have that $\h$ is a subfamily of $\h_{n,\floor{\frac{k-3}{2}},4}$ apart from $e_{2\floor{\frac{k-3}{2}}+2},e_{2\floor{\frac{k-3}{2}}+3},e_1,e_2,e_3$ and the hyperedges containing both $w$ and 
$u_1$. On the other hand, there cannot exist $h\in \h\setminus C$ with $u_{2\floor{\frac{k-3}{2}}+3},u_2\in h$ nor with $u_{2\floor{\frac{k-3}{2}}+3},u_3\in h$ as in the former case
$$w,h',u_1,e_{2\floor{\frac{k-3}{2}}+3},u_{2\floor{\frac{k-3}{2}}+3},h,u_2,e_2,\dots,u_{2\floor{\frac{k-3}{2}}+2},h'',w^*,$$
while in the latter case
$$w,h',u_1,e_1,u_2,e_2,u_3,h,u_{2\floor{\frac{k-3}{2}}+3},e_{2\floor{\frac{k-3}{2}}+2},u_{2\floor{\frac{k-3}{2}}+2},\dots,e_4,u_4,h'',w^*$$
is a Berge-path of length $k$. So we have $$|\h|\le |\h_{n,\floor{\frac{k-3}{2}},4}|+5+\binom{\floor{\frac{k-3}2}}{r-2}-2\binom{\floor{\frac{k-3}2}}{r-2}<|\h_{n,\floor{\frac{k-3}{2}},4}|,$$ contradicting the assumption on $|\h|$. So we obtained that $e_{i-1},e_i$ cannot contain $u_2,u_3$
and thus so far by Summary~\ref{sum1} we have $\h^* \subseteq \h_{n,\floor{\frac{k-1}{2}},2}$.

Now let us concentrate on the hyperedges in $\h \setminus \h^*$. So $\{u_2,u_3\}$ is the unique missing interval (all other vertices of $V\setminus D_w$ are in $R_w$), and thus $\h\setminus \h^*$ contains three hyperedges: $e_1,e_2$ and $e_3$. Observe that by Claim~\ref{replace3} (i), $e_1$ and $e_3$ do not contain any $w' \in V(\h) \setminus V$. By Claim~\ref{replace2} (v), $e_1$ and $e_3$ do not contain any vertex in $R_w$.  

\smallskip 

$\bullet$ If  $e_2$ does not contain any vertex in $R_w\cup (V(\h)\setminus V)$, then we are done, since  $\h\subseteq \h_{n,\floor{\frac{k-1}{2}},2}$. 

\smallskip 

$\bullet$ If $e_2$ does contain a vertex from $R_w\cup (V(\h)\setminus V)$, then there does not exist any other hyperedge $h$ that contains both $u_2$ and $u_3$. Indeed, if $e_2$ contained $w$, then $w$ could be inserted in between $u_2$ and $u_3$ in the Berge-cycle $C$ to form a longer cycle than $C$, a contradiction. If $e_2$ contains some $w'\neq w $ from $V(\h)\setminus V$, then we can reach a contradiction as before: we would find a Berge-path of length $k$ starting with $w',e_2,u_2,h,u_3$, then going through $C$ and ending with $u_1,h',w$ as $u_1\in D_w$. 

Finally, if $e_2$ contains a replaceable $u_i$, then at least one of $u_1,u_4$ belongs to $D_{w'}$ for some $w' \in  V(\h)\setminus V$ with $w'\neq w$, since $D_{w'} \subseteq D_w$ from  Claim \ref{replace} (i) and $|D_{w} \setminus D_{w'}|\leq 1$. By symmetry, we may assume that $u_1\in D_{w'}$. Then we have a contradiction since the following Berge-path has length $k$. The Berge-path is $u_i,e_2,u_2,h,u_3,u_4,\dots$ that goes around the cycle $C$, replaces $u_i$ by $w$ and finishes with $u_1, h_{w'}, w'$, such $h_{w'}$ exists from the definition of $D_{w'}$. Therefore, if $e_2$ does contain a vertex from $R_w\cup (V(\h)\setminus V)$, then there does not exist any other hyperedge $h$ that contains both $u_2$ and $u_3$. Hence, $\h\subseteq \h^+_{n,\floor{\frac{k-1}{2}}}$ with $A=D_w$, $L = V(\h) \setminus D_w$ and $e_2$ being the unique hyperedge of $\h^+_{n,\floor{\frac{k-1}{2}}}$ that contains less than $r-1$ vertices of $A$.

\medskip

\textsc{\textbf{Case IV/B}} For all $w'\in V(\h)\setminus V$, we have $|D_w'|=|D_w|=\floor{\frac{k-3}{2}}$.

\smallskip

As the length of $C$ is $k-1$, $k$ is even  and vertices of $D_w$ are not neighbors on $C$, we have at most three missing intervals. If there are three missing intervals, then each of them contains two vertices. If there are two missing intervals, then they contain two and three vertices and if there is only one missing interval, then it contains 4 vertices. According to this structure, we are going to consider the following three subcases.

\smallskip

-

\medskip 

\textsc{\textbf{Case IV/B/1}} There exists $w \in V(\h) \setminus V$ with $V\setminus D_w$ containing 3 intervals of length 2.

\smallskip 

Observe that as all the missing intervals are of size 2, we do not have type 1 hyperedges $h\in \h\setminus C$. As all vertices in $V\setminus D_w$ have neighbors in $D_w$, we obtain that for any $w'\in V(\h)\setminus V$ we have $D_w=D_{w'}$. So Claim~\ref{2intervals} (i) implies that there does not exist any type 2 hyperedges $h\in \h\setminus C$. Finally, Claim~\ref{replace2} (v) implies that defining hyperedges of $C$, apart from those in $\h\setminus \h^*$, are in $\h_{n,\floor{\frac{k-3}{2}},2,2,2}$. So we obtained a contradiction as $$|\h|\le 9+|\h_{n,\floor{\frac{k-3}{2}},2,2,2}|<|\h_{n,\floor{\frac{k-3}{2}},4}|.$$

\medskip

\textsc{\textbf{Case IV/B/2}} For all $w\in V(\h)\setminus V$, the number of missing intervals is at most 2 and there exist $w,w' \in V(\h) \setminus V$ with $D_w\neq D_{w'}$.

\smallskip

By relabeling, we can assume that $\{ u_2,u_3 \}$ forms the unique missing interval for both $w$ and $w'$, i.e., the unique interval of length more than 1 in $V\setminus (D_w\cup D_{w'})$. 
According to Summary \ref{sum2}, if every $u_i\notin D_w\cup D_{w'}\cup \{u_2,u_3\}$ is replaceable, then we have $\h \setminus \{ e_1,e_2,e_3\} \subseteq \h_{n,\floor{\frac{k-1}{2}},2}$, while if there is $u_i\in V\setminus (D_w\cup D_{w'})$ ($i\neq 2,3$) that is not in $R_w \cup R_{w'}$, then we know $e_{i-1}\setminus \{u_i\},e_i\setminus \{u_i\}\subseteq D_w\cup D_{w'}\cup\{u_2,u_3\}$.

\smallskip 

$\bullet$ At first we suppose that there exists a $u\in D_w\cup D_{w'}$ such that $|\{w^*\in (V(\h)\setminus V)\cup R_w\cup R_{w'}: u\in N_{\h\setminus C}(w^*)\}|=1$. In that case, the unique $w^*$ must be either $w$ or $w'$, say $w$. Consider the hypergraph $\h\setminus \h_{n,\floor{\frac{k-3}{2}},3}$ with $\h_{n,\floor{\frac{k-3}{2}},3}$ having $A=D_w\cup D_{w'}\setminus \{u\}=D_{w'}$ and $B_1=\{u,u_2,u_3\}$. Then, by Summary \ref{sum2}, the hyperedges left are incident with the vertex $u$, thus the number of hyperedges is at most $\binom{\floor{\frac{k-3}{2}}}{r-2}+5$. Here the first term is an upper bound for those hyperedges that are incident with  both $u$ and $w$, while the second term is 5 for $\{e_{i-1},e_i,e_1,e_2,e_3\}$. 
So we have a contradiction as $$|\h|\le \binom{\floor{\frac{k-3}{2}}}{r-2}+5+|\h_{n,\floor{\frac{k-3}{2}},3}|<|\h_{n,\floor{\frac{k-3}{2}},4}|.$$

\smallskip 

$\bullet$ Suppose now that for all $u\in D_w\cup D_{w'}$, $|\{w^*\in (V(\h)\setminus V)\cup R_w\cup R_{w'}: u\in N_{\h\setminus C}(w^*)\}|\geq 2$. 
At first we show that  $u_2,u_3\notin e_{i-1},e_i$ if $u_i\in V\setminus (D_w\cup D_{w'})$ ($i\neq 2,3$). This holds by Summary \ref{sum2}, if $u_i$ is replaceable by either $w$ or $w'$.
Therefore without loss of generality we may assume $u_{i+1}\in D_w\setminus D_{w'}$ and $u_{i-1}\in D_{w'}\setminus D_w$.
Note that $D_w=(D_w\cup D_{w'})\setminus \{u_{i-1}\}$ and $D_{w'}=(D_w\cup D_{w'})\setminus \{u_{i+1}\}$. 
Because of symmetry, it is enough to show a contradiction only if $u_2 \in e_i$, the three remaining cases are similar to this one. The following is a Berge-path of length $k$

\[
u_i,e_i,u_2,e_2,u_3,e_3,\dots,u_{i-1},h,w',h',u_1,e_{k-1},u_{k-1},e_{k-2}\dots, e_{i+1},u_{i+1},h'',w,
\]  
a contradiction. The hyperedges $h,h',h''$  can be chosen distinct as $u_1,u_{i-1}\in D_{w'}$ and $u_{i+1}\in D_w$ and by Lemma \ref{structure} (i), $h^*\in \h\setminus C$ cannot contain distinct vertices from outside $V$.

By Claim~\ref{replace3} (i) and (ii), $e_1$ and $e_3$ are not incident with vertices in $V(\h) \setminus V$ or in $R_w \cup R_{w'}$. Even more, they are not incident with $u_i$ either, since otherwise if $u_i \in e_1$, the following path is of length $k$, a contradiction: 
\[
u_i,e_1,u_2,e_2,u_3,e_3,\dots,u_{i-1},h,w',h',u_1,e_{k-1},u_{k-1},\dots, e_{i+1},u_{i+1},h'',w.
\]  
An analogous argument shows $u_i \notin e_3$. 

Finally, if $e_2$ contains any vertex from $V(\h)\setminus (D_w\cup D_{w'})$, then similarly to previous cases a hyperedge $e_2\neq h\in \h$ containing both $u_2,u_3$ would lead to a Berge-path of length $k$. So if no such hyperedge $h$ exists, then $\h\subseteq \h_{n,\floor{\frac{k-1}{2}},2}$. Otherwise, we have $\h\subseteq \h_{n,\floor{\frac{k-1}{2}}}^+$. Both possibilities are as claimed by the theorem.

\smallskip

\textsc{\textbf{Case IV/B/3}} For all $w'\in V(\h)\setminus V\cup R_w$, the number of missing intervals is at most 2 and  $D_w= D_{w'}$.

\smallskip

As $D_w=D_{w'}$ for all $w,w' \in V(\h) \setminus V$, it follows that we do not have to distinguish between vertices in $V(\h)\setminus V$ and vertices in $R_w$. Also, anything that we prove for hyperedges $h\in \h\setminus C$ is valid for all $e_i,e_{i-1}$ if $u_i\in R_w$, by Claim \ref{replace} (ii). 

\smallskip 

\textsc{\textbf{Case IV/B/3/1}}  Let us consider first the case when for every $v\in V(\h)\setminus V \cup R_w$,  the missing intervals for $v$ are $\{ u_2,u_3,u_4 \}$ and $\{ u_i,u_{i+1} \}$ for some $6\le i\le k-2$, after possible relabeling. By Summary \ref{sum1} and Claim~\ref{2intervals} (i), we need to consider the 7 hyperedges in $\h\setminus \h^*$, the hyperedges in $\h \setminus C$ containing $u_3,u_i$ or $u_3,u_{i+1}$ and the hyperedges in $\h \setminus C$ containing $u_3$ and some $v\in V(\h)\setminus V\cup R_w$. 

\medskip 

$\bullet$ If there are no hyperedges in $\h \setminus C$ containing $u_3,u_i$ or $u_3,u_{i+1}$ or  $u_3$ and some $v\in V(\h)\setminus V\cup R_w$,  then $\h^*\subseteq \h_{n,\floor{\frac{k-3}{2}},3,2}$, with embedding $A=D_w$, $B_1=\{u_2,u_3,u_4\}$, $B_2=\{u_i,u_{i+1}\}$ and $$|\h|\le |\h^*|+7\le |\h_{n,\floor{\frac{k-3}{2}},3,2}|+7<|\h_{n,\floor{\frac{k-3}{2}},4}|,$$ contradicting the assumption on $|\h|$. 

\smallskip 

$\bullet$ 
If there are no hyperedges in $\h \setminus C$ containing $u_3$ and some $v\in V(\h)\setminus V\cup R_w$, but there exist a hyperedge $h\in \h \setminus C$ containing $u_3,u_i$ or $u_3,u_{i+1}$, then by Claim~\ref{2intervals} (ii), there is no hyperedge containing both $u_2$ and $u_4$. In particular, with embedding $A=D_w$, $B_1=\{u_2,u_3,u_4\}$, $B_2=\{u_i,u_{i+1}\}$ we have $|\h_{n,\floor{\frac{k-3}{2}},3,2}\setminus \h|\ge \binom{\floor{\frac{k-3}{2}}}{r-2}$.
Also, by Summary \ref{sum1}, the hypergraph $\h \setminus \h_{n,\floor{\frac{k-3}{2}},3,2}$ may contain the $7$ hyperedges of $\h\setminus \h^*$ and at most $2\binom{\floor{\frac{k-3}{2}}}{r-2}+\binom{\floor{\frac{k-3}{2}}}{r-3}$ hyperedges containing $u_i$ or/and $u_{i+1}$ and $u_3$. So we have
\[
|\h|\le |\h_{n,\floor{\frac{k-3}{2}},3,2}|+7+2\binom{\floor{\frac{k-3}{2}}}{r-2}+\binom{\floor{\frac{k-3}{2}}}{r-3}-\binom{\floor{\frac{k-3}{2}}}{r-2}<|\h_{n,\floor{\frac{k-3}{2}},4}|,
\]
which contradicts the assumption on $|\h|$. 
\smallskip

$\bullet$ Suppose that there is a hyperedge $h \in \h\setminus C$ containing $u_3$ and some $v\in V(\h)\setminus V\cup R_w$. There is no $h'\in \h\setminus C$ incident with $u_2$ and $u_4$. 
Indeed, otherwise $$v,u_3,e_2,u_2,h',u_4,e_4, \dots,u_1,h_w,w$$ is a Berge-path of length $k$, a contradiction.   

By the above, Summary \ref{sum1} and Claim \ref{2intervals} (i), we have that $\h^* \subseteq \h_{n,\floor{\frac{k-1}{2}},2}$  with embedding $A=D_w\cup\{u_3\}$, $B_1=\{u_i,u_{i+1}\}$. Even more, since $D_v=D_w\ni u_1,u_5$, by Lemma \ref{structure} (iii) there exist cycles $C_2,C_4$ with $v$ replacing $u_2$ and $u_4$, respectively. Observe that the set $D_{w^*}$ does not change when we apply these changes from $C$ to $C_2$ and $C$ to $C_4$. In $C_2$, $e_1,e_2$ are not defining hyperedges, while in $C_4$, $e_3,e_4$ are not defining hyperedges. Therefore, applying Lemma \ref{structure} (ii), we obtain that $e_1,e_2$ do not contain $u_4,u_i,u_{i+1}$ and $e_3,e_4$ do not contain $u_2,u_i,u_{i+1}$. Hence hyperedges $e_1,e_2,e_3,e_4$ are also from $\h_{n,\floor{\frac{k-1}{2}},2}$ by Summary \ref{sum1}. By Claim \ref{replace3} (i) and Claim~\ref{2intervals}~(i), we have that the hyperedges $e_{i-1}$ and $e_{i+1}$ are also from $\h_{n,\floor{\frac{k-1}{2}},2}$. Finally, if $e_i$ does not contain any vertex from $(V(\h)\setminus V)\cup R_w\cup \{u_3\}$, then we have $\h\subseteq \h_{n,\floor{\frac{k-1}{2}},2}$. Otherwise, as in Case IV/A, one can see that there does not exist $h\neq e_i$ with $u_i,u_{i+1}\in h$ and thus $\h=\h_{n,\floor{\frac{k-1}{2}}}^+$ with $A=D_w\cup\{u_3\}$ and  $e_i$ being the unique hyperedge with less than $r-1$ elements in $A$.

\smallskip 

\textsc{\textbf{Case IV/B/3/2}} For all $v\in V(\h)\setminus V\cup R_w$, the only missing interval consists of $\{ u_2,u_3,u_4,u_5 \}$, after possible relabelling.

\smallskip

By Summary \ref{sum1}, we need to handle hyperedges $e_1,e_2,e_3,e_4,e_5$ and those $h\in \h \setminus C$  that contain a $v\in V(\h)\setminus V \cup R_w$ and $u_3$ and/or $u_4$. 

\smallskip 

$\bullet$ If there are no such hyperedges and $e_1,e_2,e_3,e_4,e_5 \subseteq D_w\cup \{u_2,u_3,u_4,u_5\}$, then $\h\subseteq \h_{n,\floor{\frac{k-3}{2}},4}$ contradicting the assumption on $|\h|$.

\smallskip 

$\bullet$ Suppose next there is no $h\in \h\setminus C$ with a vertex from $V(\h) \setminus V \cup R_w$ containing $u_3$ or $u_4$, but some $e_i$ ($i=1,2,3,4,5$) does contain a vertex from outside $V$. By Claim~\ref{replace3} (i), it is  neither $e_1$ nor $e_5$. If $e_i$ contains a vertex $v$ from outside $V$, then there cannot exist $h\in \h\setminus C$ with $u_2,u_{i+1}\in h$, as then
$$v,e_i,u_i,e_{i-1},\dots,u_2,h,u_{i+1},e_{i+1},u_{i+2},e_{i+2},\dots,u_{k-1},e_{k-1},u_1,h'$$ is a Berge-cycle of length $k$. For the existence of $h'$ we used $D_v=D_w\ni u_1$. Therefore we have $$|\h|\le 3+|\h_{n,\floor{\frac{k-3}{2}},4}|-\binom{\floor{\frac{k-3}{2}}}{r-2}<|\h_{n,\floor{\frac{k-3}{2}},4}|,$$ contradicting the assumption on $|\h|$.

$\bullet$  If there exists  a hyperedge $h \in \h\setminus C$ incident with some vertex $v \in V(\h)\setminus V \cup R_w$ and $u_3$,  then there is no $h'\neq h$,  $h' \in \h \setminus C$ incident with some vertex from $V(\h)\setminus V \cup R_w$ and $u_4$,  by Claim~\ref{replace} (i). 
Even more, there  is no $h''\in \h\setminus C$ with $u_2,u_4\in h''$. The argument is the same as if $e_3$ contained $v$ from the previous bullet. Similarly one can get that there is no hyperedge  $h''\in \h\setminus C$ with $u_2,u_5\in h''$. 

Observe that there should exist at least two distinct $v_1,v_2\in V(\h)\setminus V \cup R_w$ for which hyperedges $h_{v_1},h_{v_2}$ with $v_1,u_3\in h_{v_1}$ and $v_2,u_3\in h_{v_2}$ %(GDANI: what is that?) 
exist. Indeed, otherwise using that there is no non-defining edge incident to $u_2,u_4$, we have
\[
|\h|\le 5+1+|\h_{n,\floor{\frac{k-3}{2}},4}|-\binom{\floor{\frac{k-3}{2}}}{r-2}<|\h_{n,\floor{\frac{k-3}{2}},4}|.
\]

We will show that either $\h\subseteq \h_{n,\floor{\frac{k-1}{2}},2}$ or $\h\subseteq \h_{n,\floor{\frac{k-1}{2}}}^+$ with $A=D_w\cup \{u_3\}$ and $B_1=\{u_4,u_5\}$. Let 
$w^*$ denote an arbitrary vertex $w^*\in V(\h)\setminus V$ with $w^*\neq v_1,v_2$. We will use that $u_1,u_6\in D_w=D_{w^*}=D_{v_1}=D_{v_2}$ thus there exists a hyperedge that is not a defining hyperedge of $C$ and is different from $h_{v_1}$ and $h_{v_2}$, containing either $u_1$ or $u_6$ together with $v_1$ or $v_2$ or $w^*$.

 We need to prove that $u_4,u_5\notin e_1,e_2$ and $u_2\notin e_3,e_5$. In each of the cases, we present a Berge-path of length $k$ below, which is a contradiction.

\medskip 

If $u_4\in e_1$, then the path is $$v_1,h_{v_1},u_3,e_2,u_2,e_1,u_4,e_4,u_5,\dots,u_{k-1},e_{k-1},u_1,h,w^*.$$

If $u_4\in e_2$, then the path is
$$u_2,e_2,u_4,e_4,u_5,e_5,\dots,u_{k-1},e_{k-1},u_1,h,v_1,h_{v_1},u_3,h_{v_2},v_2.$$

If $u_5\in e_1$ or $e_2$ , then the path is $$u_2,e_1 \mbox { or } e_2,u_5,e_4,u_4,e_3,u_3,h_{v_1},v_1,h,u_6,e_6,\dots,u_{k-1},e_{k-1},u_1,h',w^*.$$ 

If $u_2\in e_3$, then the path is $$v_1,h_{v_1},u_3,e_2,u_2,e_3,u_4,e_4,u_5,\dots,u_{k-1},e_{k-1},u_1,h,w^*.$$

If $u_2\in e_5$, then the path is $$u_2,e_5,u_5,e_4,u_4,e_3,u_3,h_{v_1},v_1,h,u_6,e_6,\dots,u_{k-1},e_{k-1},u_1,h',w^*.$$

 From here, one can conclude to $\h\subseteq \h_{n,\floor{\frac{k-1}{2}},2}$ or $\h\subseteq \h_{n,\floor{\frac{k-1}{2}}}^+$ as in Case IV/A, depending on whether $e_4\subseteq D_w\cup \{u_3,u_4,u_5\}$ or not. 

\bigskip 
   
The above case-by-case analysis concludes the proof of Theorem~\ref{deg2} under the set degree condition, i.e., for any set $X$ of vertices with $|X|\le k/2$ the number of hyperedges  incident with some vertex in $X$, $|E(X)|$, is at least $|X|\binom{\lfloor \frac{k-3}{2} \rfloor }{r-1}$.
\end{proof}

\section{Proof of Theorem~\ref{deg2} and Theorem~\ref{general}}
Let $n'_{k,r}$ denote the threshold such that the statement of Theorem 6 holds for hypergraphs with the set degree condition if $n\ge n'_{k,r}$. We are now ready to prove the general statements.

\begin{proof}[Proof of Theorem~\ref{deg2} and Theorem~\ref{general}]
\bigskip
 Let $\h$ be a connected  $n$-vertex  $r$-uniform hypergraph without a Berge-path of length $k$, and  suppose that if $k$ is odd, then
\begin{equation*}
    |\h|>|\h_{n,\lfloor \frac{k-3}{2}\rfloor,3}| =\left(n- \frac{k+3}{2}\right) \binom{\lfloor \frac{k-3}{2} \rfloor }{r-1}+\binom{\lfloor \frac{k+3}{2} \rfloor}{r},
\end{equation*}
while if $k$ is even, then 
\begin{equation*}
    |\h|>|\h_{n,\lfloor \frac{k-3}{2}\rfloor,4}| =\left(n- \floor{\frac{k+5}{2}}\right) \binom{\lfloor \frac{k-3}{2} \rfloor }{r-1}+\binom{\lfloor \frac{k+5}{2} \rfloor}{r}.
\end{equation*}

 We obtain a sub-hypergraph $\h'$ of $\h$ using a standard greedy process: as long as there exists a set $S$ of vertices with $|S|\le k/2$ such that $|E(S)|< |S| \binom{\floor{\frac{k-3}{2}}}{r-1}$, we remove $S$ from $\h$ and all hyperedges in $E(S)$. Let $\h'$ denote the sub-hypergraph at the end of this process.
 \begin{claim}\label{subhyp}
 There exists a threshold 
 $n''_{k,r}$,  such that if $|V(\h)| \ge n''_{k,r}$, then $\h'$ is connected and contains at least $n'_{k,r}$ vertices.
 \end{claim}
 
 \begin{proof}
 To see that $\h'$ is connected, observe that every component of $\h'$ possesses the set degree condition. Therefore Claim~\ref{Contains_cycle} yields that every component contains a cycle of length at least $k-4$. Therefore, as $\h$ is connected, $\h$ contains a Berge-path with at least $2k-8$ vertices from two different components of $\h'$, a contradiction as $k\ge 9$.
 
 Suppose to the contrary that  $\h'$ has less than $n'_{k,r}$ vertices. 
 Observe that, by definition of the process, $|\E(\h')|-|V(\h')|\binom{\floor{\frac{k-3}{2}}}{r-1}$ strictly increases at every removal of some set $X$ of at most $k$ vertices. Therefore if $n>n'_{k,r}+k\binom{n'_{k,r}}{r}=n''_{k,r}$ and $|V(\h')|<n'_{k,r}$, then  at the end we would have more hyperedges than those in the complete $r$-uniform hypergraph on $|v(\h')|$ vertices, a contradiction.  
 \end{proof}

 By Claim~\ref{subhyp} and the statement for hypergraphs with the set degree property, we know that  $\h'$ has $n_1\ge n'_{k,r}$ vertices,  and $\h'\subseteq \h_{n_1,\lfloor \frac{k-1}{2}\rfloor}$ if $k$ is odd, and $\h'\subseteq \h_{n_1,\lfloor \frac{k-1}{2}\rfloor,2}$ or $\h^+_{n_1,\floor{\frac{k-1}{2}}}$ if $k$ is even. Then for any hyperedge $h\in \E(\h)\setminus \E(\h')$ that contains at least one vertex from $V(\h)\setminus V(\h')$ with the degree at least two, we can apply Lemma~\ref{structure} (i) to obtain that all such $h$ must meet the $A$ of $\h'$ in $r-1$ vertices. This shows that if the minimum degree of $\h$ is at least 2, then $\h\subseteq \h_{n_2,\lfloor \frac{k-1}{2}\rfloor}$ if $k$ is odd, and $\h\subseteq \h_{n_2,\lfloor \frac{k-1}{2}\rfloor,2}$ or $\h \subseteq \h^+_{n_2,\floor{\frac{k-1}{2}}}$ if $k$ is even, where $n_2\le n$ is the number of vertices that are contained in a hyperedge of $\h$ that is either in $\h'$ or has a vertex in $V(\h)\setminus V(\h')$ with the degree at least 2. This finishes the proof of Theorem~\ref{deg2}.
 
 Finally, consider the hyperedges that contain the remaining $n-n_2$ vertices. As all these vertices are of degree 1, they are partitioned by these edges. For such a hyperedge $h$ let $D_h$ denote the subset of such vertices.  Observe that for such a hyperedge $h$, we have that $h\setminus D_h \subseteq A$. Indeed if $v\in h\setminus (D_h\cup A)$, then there exists a cycle $C$ of length $k-1$ in $\h'$ not containing $v$. Thus there is a path of length at least $k$ starting at an arbitrary $d\in D_h$, continuing with $h,v$, and having $k-1$ more vertices as it goes around $C$ with defining vertices and hyperedges a contradiction. 
%Claim \ref{Contains_cycle} and finishes the proof of Theorem~\ref{general}.
 \end{proof}
 
 \textbf{Acknowledgement.} 
We would like to express our sincere appreciation to the anonymous referees for their valuable comments, which significantly improved the quality of the paper's presentation.
 
 Research partially sponsored by the National Research, Development and Innovation Office NKFIH under the grants K 116769, K 132696, KH 130371, SNN 129364, PD 137779 and FK 132060. 
 Research of Gerbner, Nagy and Vizer was supported by the J\'anos Bolyai Research Fellowship of the Hungarian Academy of Sciences. Patk\'os acknowledges the financial support from the Ministry of Educational and Science of the Russian Federation in the framework of MegaGrant no 075-15-2019-1926. 
 Research of Vizer was supported by the New National Excellence Program under the grant number \'UNKP-20-5-BME-45.


\begin{thebibliography}{1}
\bibitem{BGLS}
P.N. Balister, E. Gy{\H{o}}ri, J. Lehel, R.H. Schelp. Connected graphs without long paths. {\it Discrete Mathematics} \textbf{308}(19), (2008) 4487--4494.

\bibitem{DavoodiGMT}
A. Davoodi, E. Gy\H ori, A. Methuku, C. Tompkins. An Erd\H{o}s-Gallai type theorem for uniform hypergraphs. {\it European Journal of Combinatorics} \textbf{69}, (2018) 159--162.

\bibitem{Er-Ga} P. Erd\H os,  T. Gallai.  On maximal paths and circuits of graphs. {\it Acta Math. Acad. Sci. Hungar.}  \textbf{10}, (1959) 337--356.

\bibitem{FSch}
R.J. Faudree, R.H. Schelp. Path Ramsey numbers in multicolorings. \textit{Journal of Combinatorial Theory, Series B}, \textbf{19}(2), (1975) 150--160.

\bibitem{furedi2018avoiding}
Z. F\"uredi, A. Kostochka, R. Luo. Avoiding long Berge-cycles. {\it Journal of Combinatorial Theory, Series B}, {\bf 137}, (2019) 55--64.

\bibitem{furedi2019connpath}
Z. F\"uredi, A. Kostochka, R. Luo. On 2-connected hypergraphs with no long cycles. {\it Electronic Journal of Combinatorics}, {\bf 26}, (2019) P.4.31.

\bibitem{furedi2016stability} Z. F\"uredi, A. Kostochka, J. Verstra\"ete. Stability Erd\H os–Gallai Theorems on cycles and paths. {\it Journal of Combinatorial Theory, Series B}, \textbf{121}, (2016) 197--228.

\bibitem{FurSim}
Z. F\"uredi, M. Simonovits. The history of degenerate (bipartite) extremal graph problems. In Erdős Centennial (pp. 169--264). Springer, Berlin, Heidelberg. (2013)


\bibitem{GyoKaLe} E. Gy\H ori, G. Y. Katona, N. Lemons. Hypergraph extensions of the Erd\H os-Gallai
 Theorem. {\it European Journal of Combinatorics} \textbf{58}, (2016) 238--246.

\bibitem{GLSZ}
E. Gy{\H{o}}ri, N. Lemons, N. Salia, O. Zamora. The Structure of Hypergraphs without long Berge-cycles.  {\it Journal of Combinatorial Theory, Series B}, \textbf{148} (2021) 239--250.

\bibitem{GyoriMSTV}
E. Gy\H ori, A. Methuku, N. Salia, C. Tompkins, M. Vizer. On the maximum size of connected hypergraphs without a path of given length. 
{\it Discrete Mathematics} \textbf{341}(9), (2018) 2602--2605.

\bibitem{GSZ} E. Gy\H ori, N. Salia, O. Zamora. Connected Hypergraphs without long paths. {\it European Journal of Combinatorics} \textbf{96} (2021) 103353.

\bibitem{Kopylov} G.N. Kopylov. On maximal paths and cycles in a graph. {\it Soviet Math.} (1977) 593--596.



\end{thebibliography}
\end{document}